\documentclass{amsart}
\usepackage{amssymb, amsthm}
\usepackage{mathtools}
\usepackage{hyperref}
\allowdisplaybreaks

\newcommand{\R}{{\mathbb{R}}}

\newcommand{\N}{{\mathbb{N}}}

\newcommand{\qtq}[1]{\quad\text{#1}\quad}
\newcommand{\eps}{\varepsilon}

\newtheorem{theorem}{Theorem}[section]
\newtheorem{lemma}[theorem]{Lemma}
\newtheorem{proposition}[theorem]{Proposition}
\newtheorem{corollary}[theorem]{Corollary}
\theoremstyle{definition}
\newtheorem{definition}[theorem]{Definition}

\numberwithin{equation}{section}

\renewcommand{\H}{\mathcal H}
\newcommand{\la}{\langle}
\newcommand{\ra}{\rangle}

\title[The Kirchhoff--Pohozaev wave equation]{Existence and equicontinuity of solutions\\to the Kirchhoff--Pohozaev wave equation}
\author[L. Campos]{Luccas Campos}
\address{Department of Mathematics, Universidade Federal de Minas Gerais, Brazil}
\email{luccas@mat.ufmg.br}

\author[R. Killip]{Rowan Killip}
\address{CEREMADE, CNRS, Universit\'e Paris Dauphine\textbf{--}PSL, 75016 Paris, France}
\email{killip@ceremade.dauphine.fr}

\author[M. Visan]{Monica Vi\c{s}an}
\address{Institute of Science and Technology, Am Campus 1, 3400 Klosterneuburg, Austria}
\email{Monica.Visan@ist.ac.at}

\begin{document}

\begin{abstract}
It was recently discovered by Boiti--Manfrin that a variant of the Kirchhoff wave equation introduced by Pohozaev admits infinitely many conserved quantities.
In this paper, we obtain explicit formulae for such conserved quantities, as well as introduce a generating function for them that is coercive.  These tools are then employed to  establish a priori bounds, the propagation of equicontinuity, global well-posedness in $\mathcal H^s$ for $s\geq \frac32$, and the existence of global $C_t \mathcal H^1$ solutions.  
\end{abstract}

\maketitle
\section{Introduction}

Waves on a stretched string can be modeled by the linear wave equation with wave speed $c$ given by $c^2=\tau/\rho$, where $\tau$ is the tension of the string and $\rho$ its mass density.  In his lectures \cite{Kirchhoff97} on mathematical physics, Kirchhoff took the analysis further by incorporating the following nonlinear effect: when the string is displaced from equilibrium, it is stretched further and this increases the tension.

In non-dimensionalized form, Kirchhoff's dynamical equation reads
\begin{equation}\label{E:Kirchhoff}
\partial_{t}^2 u(t,x) - \biggl(1+ \int_U |\nabla u(t,x')|^2\, dx'\biggr)  \Delta u(t,x)=0,
\end{equation}
where $\Delta$ denotes the Dirichlet Laplacian on the interval $U=[0,\ell]$ and $\nabla$ denotes the (single) spatial derivative.
With very mild physical assumptions/approximations, Kirchhoff has provided us with a very elegant dynamical equation that preserves the Hamiltonian character of wave motion.  Despite concerted effort over many years, the foundational question of the well-posedness of \eqref{E:Kirchhoff} in the energy space remains open!

Early investigations of \eqref{E:Kirchhoff} include \cite{Bernstein,Dickey,Lions,Pohozaev75}, where the model was also generalized in a number of ways.  First, vibrating membranes (and higher-dimensional analogues) were considered by allowing general domains $U\subseteq\R^d$; this is the reason why we adopted dimension-independent notations in \eqref{E:Kirchhoff}.  Further, general tension functions were considered and it was realized that much could be achieved without the direct link to wave equations: the operator $-\Delta$ could be replaced by a general positive definite operator $A$ on an abstract Hilbert space $X$.  In this way, the study of \eqref{E:Kirchhoff} has grown into the consideration of the family of systems
\begin{align}\label{GenK_ODE}
\tfrac{d}{dt}q  &= p, \qquad \tfrac{d}{dt}p = - M'\bigl(\langle q,A q\rangle\bigr) A q ,
\end{align}
where the function $M:[0,\infty)\to [0,\infty)$ is increasing.  Such systems admit the Hamiltonian
\begin{align}\label{GenK_H}
H = \tfrac12 \langle p, p\rangle + \tfrac12 M\bigl(\langle q,A q\rangle\bigr)
\end{align}
with respect to the standard Poisson bracket
\begin{align}\label{Poisson}
\{F, G\} = \bigl\langle \tfrac{\partial F}{\partial q},\tfrac{\partial G}{\partial p} \bigr\rangle -\bigl\langle \tfrac{\partial F}{\partial p},\tfrac{\partial G}{\partial q} \bigr\rangle.
\end{align}

In the context of the wave equation \eqref{E:Kirchhoff}, we have $A=-\Delta$, $q=u$, $p=\partial_t u$, and $M(s) =s+\frac12 s^2$. When working with the wave equation in the whole space $\R^d$, the equation \eqref{GenK_ODE} admits the scaling symmetry
\begin{equation}\label{scaling}
q(t,x)\mapsto\mu^{(d-2)/2}q(\mu t,\mu x) \qtq{and} p(t,x)\mapsto\mu^{d/2}p(\mu t,\mu x).
\end{equation}
Notably, this scaling also preserves the Hamiltonian \eqref{GenK_H}, irrespective of the choice of $M$.  Indeed, the same phenomenology can be observed whenever $A$ admits a scaling symmetry. Correspondingly, we call the whole family of models \emph{energy critical}.

In \cite{Pohozaev85}, Pohozaev identified a new conservation law, besides the Hamiltonian, when $M'(s)$ takes the form $(c_1 s +c_2)^{-2}$.  Using that $M'\geq 0$ and rescaling, one may reduce attention to the two cases
\begin{align}\label{KP_pm}
\tfrac{d}{dt}q  &= p, \qquad \tfrac{d}{dt}p = -c_\pm^2(q) A q \qtq{with} c_\pm = (1 \pm \langle q,A q\rangle\bigr)^{-1} .
\end{align}
The corresponding Hamiltonians are
\begin{equation}\label{Hamiltonian}
H_\pm(q,p) = \tfrac12 \Bigl[\langle p, p\rangle + \tfrac{\langle q, A q\rangle}{1 \pm \langle q, A q\rangle}\Bigr],
\end{equation}
and Pohozaev's new conservation law takes the form
\begin{align}\label{I Poho}
I_1^\pm(q,p) :=  \langle p, A p\rangle + \tfrac{\langle q, A^{2} q\rangle}{1 \pm \langle q, A q\rangle}
	\pm \Bigl[ \langle q, A q\rangle \langle p, A p\rangle - \langle p, A q\rangle^2 \Bigr].
\end{align}

The model generated by $H_-(q,p)$ is only studied in the regime $\langle q, A q\rangle < 1$. Due to the divergence of the Hamiltonian, larger values of $\la q, Aq\ra$ are inaccessible via finite-energy excitations of the equilibrium configuration $p=q=0$.

Our own interest in the Kirchhoff--Pohozaev models \eqref{KP_pm} was stimulated by the recent paper \cite{BM26} where, building on \cite{BM25}, it was shown that these models actually admit \emph{infinitely many} conservation laws.  This prompted us to consider what might be achieved toward the well-posedness of the models \eqref{KP_pm} by deploying recent approaches to the study of integrable systems. 

To be considered completely integrable, it is not merely enough that a model has sufficiently many conserved quantities; these observables must Poisson commute!\footnote{We would also like to draw the reader's attention to the very recent paper \cite{HM} that seeks to understand the complete integrability of the Kirchhoff--Pohozaev models through the lens of Birkhoff normal form transformations.}  Before addressing this question, we first need to find explicit expressions for the conserved quantities.  We summarize these achievements as follows:

\begin{theorem}\label{T:1.1}
For each model \eqref{KP_pm}, smooth solutions conserve
\begin{equation}\label{E:Iz combined}
\begin{aligned}
I_\pm(z;q,p)&=\langle p,\tfrac{1}{1+zA}p\rangle+c_\pm(q)\langle q,\tfrac{A}{1+zA}q\rangle\\
&\qquad \mp z\bigl[\langle q,\tfrac{A}{1+ zA}q\rangle\langle p,\tfrac{A}{1+ zA}p\rangle-\langle p,\tfrac{A}{1+ zA}q\rangle^2\bigr].
\end{aligned}
\end{equation}
Moreover, $I_\pm(z)$ and $I_\pm(w)$ Poisson commute for all $z,w\geq 0$.
\end{theorem}

Note that the quantities \eqref{E:Iz combined} serve as generating functions for the more traditional `polynomial' conservation laws, whose explicit form may be derived by expanding $(1+zA)^{-1}$ in powers of $z$; see, for example, \eqref{Ik} and \eqref{expansion}.  Note that Poisson commutativity of the generating functions guarantees the Poisson commutativity of these polynomial conservation laws.

It is natural to ask if our conserved quantities agree with those in \cite{BM26}.  This is a little subtle because those authors merely prove that there is a conserved quantity with prescribed leading term.  While our leading term agrees with theirs, such conservation laws are not unique: one may add linear (or even polynomial) combinations of lower-indexed conserved quantities.  We contend that our polynomial conservation laws $I_k(q,p)$ are the canonical ones because (a) each is (at most) quadratic in the momentum variable and (b) they are all scaling homogeneous.

To clarify what we mean by scaling homogeneous, let us return to the case $A=-\Delta$ acting on $\R^d$, which admits the symmetry \eqref{scaling}.  It is elementary to verify that under this transformation, $I_k \mapsto \mu^{2k} I_k$.

To be useful in addressing the well-posedness problem, conservation laws need to be coercive, that is, they need to provide control on norms of the solution!  It is already instructive to examine the Hamiltonians \eqref{Hamiltonian} from this perspective.  In the $-$ case where solutions satisfy $\la q, Aq\ra<1$, the Hamiltonian $H_-(q,p)$ constrains the $\H^1$ norm
$$
\|(q,p)\|_{\H^1}^2:= \|A^{\frac12}q\|^2 + \|p\|^2
$$
of the solution. Indeed, the displacement $q$ is more strongly confined than for the linear wave equation. Correspondingly, we will call this the \emph{strongly confined} model.  

By contrast, in the $+$ case, the Hamiltonian $H_+(q,p)$ does not provide any upper bound on the $\H^1$ norm of the solution once $H_+(q,p)>\frac12$. Correspondingly, we will call this the \emph{unconfined} model.


Rather than jump back and forth between the two cases, we will focus primarily on the unconfined model for the rest of this paper; accordingly, the absence of a $\pm$ subscript will indicate the unconfined case.  We select this case for our analysis because it presents the greater difficulty in obtaining a priori bounds, as we have already seen in our examination of the Hamiltonian.  This point seems to have been already appreciated by Pohozaev: as the first application of the new conservation law \eqref{I Poho}, he proved a priori $\H^1$ bounds for solutions of the unconfined model in terms of the $\H^2$ norm of the initial data.  His bound grows linearly in time; this cannot be improved, as we observe in Corollary~\ref{C:93a}.

Naturally, one would like to bound the $\H^1$ norm of the solution in terms of the $\H^1$ norm of the initial data.  This is impossible, as we will demonstrate in Corollary~\ref{C:93b}.  Nevertheless, we are still able to prove $\H^1$-bounds for $\H^1$ initial data.  There is no contradiction here: our estimate is uniform on $\H^1$-\emph{precompact} sets of initial, rather than \emph{bounded} sets.  Actually, we will not need the full strength of compactness, merely boundedness and equicontinuity. Moreover, we are able to work in the broader class of $\H^s$ spaces defined in \eqref{Hs def}.

\begin{definition}[Equicontinuity]\label{D:equi}
Fix $s\geq 1$. We say that a bounded set $\Omega \subset \H^s$ is \emph{equicontinuous} in $\H^s$ if for any $\eps>0$ there exists $\kappa=\kappa(\eps)>0$ such that
\begin{equation}\label{equi}
\sup_{(q,p)\in\Omega}\, \|(P_{\geq \kappa}q, P_{\geq \kappa} p)\|_{ \H^s} \leq \eps.
\end{equation}
Here $P_{\geq \kappa}$ denotes the spectral projection \eqref{LP def}. 
\end{definition}

Evidently, \eqref{equi} expresses tightness of the spectral measures associated to $A$ and the vectors in $\Omega$. Our motivation for calling this property \emph{equicontinuity} originates from the case $A=-\Delta$.  When the operator $A$ has compact resolvent, any set $\Omega$ that is bounded and equicontinuous is automatically precompact.  This is not true in general; consider, for example, the case $A=-\Delta$ acting on $\R^d$.

Any $\H^s$-Cauchy sequence $(q_{0,n},p_{0,n})$ of initial data constitutes an $\H^s$-bounded and equicontinuous set.  Evidently, it would be advantageous to the well-posedness problem if one could not only prove that the corresponding orbits $(q_n(t),p_n(t))$ are $\H^s$-bounded at later times $t$, but also that they are equicontinuous.  This we can prove:

\begin{theorem}\label{T:1.2}
For either model \eqref{KP_pm}, every  $s\geq 1$, every $T>0$, and every set $\Omega$ of smooth initial data that is $\H^s$-bounded, $\H^s$-equicontinuous,
and satisfies
\begin{align}\label{energy_pm}
\sup_{(q,p) \in \Omega}\, H_\pm(q,p)<\infty,
\end{align}
the corresponding set of orbits
\begin{equation}\label{T:1.2a}
 \bigl\{ \bigl(q(t),p(t)\bigr) \big| \bigl(q(0),p(0)\bigr)\in\Omega \text{ and } |t|\leq T \bigr\}
\end{equation}
is likewise $\H^s$-bounded and $\H^s$-equicontinuous.
\end{theorem}

In formulating this theorem, we have been guided by unity; a more complete picture can be found in Theorems~\ref{T:equicontinuity_H1}, \ref{T:equicontinuity_Hs}, and~\ref{T:confined}, as well as Corollary~\ref{C:1.2}.

In the unconfined case, the hypothesis \eqref{energy_pm} is automatically satisfied by virtue of $\H^s$ boundedness. Moreover, we will show that norms grow at most linearly in time; see Theorems~\ref{T:equicontinuity_H1} and \ref{T:equicontinuity_Hs}.

In the strongly confined case, the role of \eqref{energy_pm} is to enforce a uniform quantitative version of the restriction $\la q, Aq\ra<1$ across $\Omega$. For this model, we do not need equicontinuity of the set of initial data to obtain bounds.  Moreover, in Theorem~\ref{T:confined} we prove bounds that are uniform in time for all $1\leq s<2$.  This covers the regime $s\in[1,\tfrac32]$ where global well-posedness requires further attention.  Rather than repeating arguments from the more difficult unconfined case to treat larger $s$, we take the opportunity to exhibit a robust general method for propagating bounds and equicontinuity; see subsection~\ref{SS:4.2}, the climax of which is Corollary~\ref{C:1.2}.

Building on Theorem~\ref{T:1.2}, we make two contributions to the well-posedness question.

\begin{theorem}\label{T:1.3}
Both models \eqref{KP_pm} are globally well-posed in $\H^{s}$ for $s\geq \frac32$.
\end{theorem}

The greatest novelty here is that we are able to work \emph{globally} in time, particularly, at the endpoint $s=3/2$.  For a broad class of tension functions, including \eqref{KP_pm}, the $\H^s$ \emph{local} well-posedness of \eqref{GenK_ODE} has been shown for $s\geq 3/2$; see \cite{AG,AP96I} and the references therein.  The global bounds announced in Theorem~\ref{T:1.2} immediately extend this to global well-posedness.  Global well-posedness for the strongly confined model when $s>3/2$ was obtained in \cite{Panizzi} through the construction of \emph{almost} conserved quantities.

\begin{theorem}\label{T:1.4}
For every initial data $(q_0,p_0)\in \H^1$, both models \eqref{KP_pm} admit a global $C_t \H^1$-solution.
\end{theorem}

We are unable to prove the uniqueness of such solutions at this time; this is the only obstruction to obtaining global well-posedness in $\H^1$.  The only previous result we are aware of regarding the existence of $C_t \H^s$ solutions with $s<3/2$ comes from \cite{GG}.  This paper constructs global $C_t \H^1$ solutions when the initial data is highly lacunary in the spectral variable; see also \cite{Manfrin_Lac} for a lacunary result in $C_t \H^2$.

\subsection*{Organization of the paper}  Throughout this introduction, we have taken the existence of global smooth solutions for granted.  This is justified by the existing works \cite{AG,Pohozaev85}.  Nevertheless, the fact that we are able to demonstrate strong quantitative bounds for smooth solutions (as well as the propagation of equicontinuity) provides an opportunity to give a presentation of well-posedness that is both streamlined and self-contained.  The structure of this paper is informed by the desire to do this while simultaneously proving Theorems~\ref{T:1.1} through~\ref{T:1.4} with a minimum of repetition.

Section~\ref{S:2} first introduces notations used throughout the paper.  We then construct global solutions for a very narrow class of initial data in Lemma~\ref{L:GWP}.  In this setting, $A$ is effectively a bounded operator and one may employ the Picard--Lindel\"of Theorem.

Section~\ref{S:3} is devoted to the proof of Theorem~\ref{T:1.1}.  We first introduce the polynomial conserved quantities \eqref{Ik} in the unconfined case and show directly that they are conserved for smooth solutions.  Next, we introduce the generating function \eqref{I+} and prove Poisson commutativity in Proposition~\ref{P:Poisson}.  The minor changes needed for the strongly confined case are discussed in subsection~\ref{SS:3.1}.

Section~\ref{S:4} culminates in the proof of Theorem~\ref{T:1.2}.  We begin by treating the unconfined model with $s=1$ in Theorem~\ref{T:equicontinuity_H1}, before treating larger $s$ in Theorem~\ref{T:equicontinuity_Hs}.  The conservation laws associated to the strongly confined model are used to treat $1\leq s<2$ in Theorem~\ref{T:confined}.  Building on the $s=3/2$ case of Theorem~\ref{T:confined}, we treat general $s\geq 3/2$ in Corollary~\ref{C:1.2} by a completely different method.

In Section~\ref{S:5}, we prove Theorem~\ref{T:1.3}.  We construct solutions by taking limits of the spectrally truncated solutions constructed in Lemma~\ref{L:GWP}; at this moment, these are the only solutions for which we may apply the results of Section~\ref{S:4} without appealing to prior work.  The restriction $s\geq 3/2$ ultimately originates from a Gronwall argument in Lemma~\ref{L:F}, which is used to demonstrate uniqueness and the continuous dependence of the solution on the initial data.  As in subsection~\ref{SS:4.2} and many prior works, this marks the threshold at which one can control the time derivative of the speed.  On the other hand, the way we use equicontinuity stemming from conservation laws to upgrade continuous dependence to the correct topology is quite unlike any prior work on these models.

Section~\ref{S:6} is devoted entirely to the proof of Theorem~\ref{T:1.4}.  Once again we take limits of spectrally truncated solutions.  Our success rests entirely on the fact that we can propagate both bounds and equicontinuity at the critical regularity $\H^1$.

The paper closes with Section~\ref{S:7}, which is devoted to the analysis of the unconfined model with initial data $(q_0,p_0)$ lying in a single eigenspace of the operator $A$.  In this setting, the problem is explicitly solvable by very traditional methods.  Nonetheless, such solutions provide a clear demonstration of some of the difficulties intrinsic to this model and justify the sharpness of our results.

\subsection*{Acknowledgements} Much of this work was completed while L.~C. and R.~K. were visiting the UCLA Mathematics Department; we are grateful for their hospitality.  L.~C. was partially supported by CNPq grants 07733/2023-8 and 404800/2024-6, and the FAPEMIG grant APQ-03752-25. R.~K. was supported by NSF grant DMS-2452346 and the project ANR-25-CFFS-0004 of the France 2030 program.  M.~V. was partially supported by NSF grant DMS-2348018.

\section{Notation and Preliminaries}\label{S:2}

Throughout the paper,  $(X, \la\cdot,\cdot\ra)$ will denote a Hilbert space and $A:\mathcal D(A) \to X$ will be nonnegative, selfadjoint operator, whose domain $\mathcal D(A)$ is dense in $X$.

For $N>0$ we define the orthogonal projections $P_{<N}$ and $P_{\geq N}$ as follows:
\begin{equation}\label{LP def}
P_{< N} := \chi_{[0,N^2)}(A) \qtq{and}  P_{\geq N}  := \chi_{[N^2, \infty)}(A)=1-P_{<N}.
\end{equation}
This notation is informed by the notion of Littlewood--Paley projections applicable to the case $A=-\Delta$.

For $s \in \R$, we define the space $H^s$ as the completion of the domain of $(1+A)^{\frac s2}$ with respect to the norm
\begin{equation}\label{H defn}
\|\phi\|_{ H^s}^2 = \la \phi, (1+A)^s\phi\ra.
\end{equation}
Note that if $s=0$, then $H^0=X$.  In the case when $A=-\Delta$ on $\R^d$, these spaces are precisely the traditional Sobolev spaces, which is our motivation for this notation.  The corresponding space of smooth functions is 
\begin{equation}\label{Hinfty}
H^{\infty} = \bigcap_{k\geq 0} H^k \qtq{with metric} d(f,g) = \sum_{k\geq 0} 2^{-k} \frac{\|f-g\|_{H^k}}{1 + \|f-g\|_{H^k}}.
\end{equation}
 
For wave equations on $\R^d$, the natural classes of initial data correspond to $\nabla_{t,x} u \in H^{s-1}(\R^d)$. The analogous spaces in our current abstract setting are defined as follows:  For $s \geq 0$,
\begin{equation}\label{Hs def}
\begin{aligned}
\H^s &:=\{(q,p)\in X\times X\, |\, A^{\frac12}q, p \in H^{s-1}\}
\end{aligned}
\end{equation}
equipped with the norm
\begin{align}\label{Hs norm}
\|(q,p)\|_{\H^s}^2:= \|A^{\frac12}q\|_{H^{s-1}}^2+ \|p\|_{H^{s-1}}^2.
\end{align}
The metric space $\H^\infty$ is defined analogously using \eqref{Hinfty}.

\begin{definition}[Solution]
Fix $s\geq 1$. For an open interval $I\subseteq \R$, we say that $(q,p)\in C(I;\H^s)\cap C^1(I; \H^{s-1})$ is an \emph{$\H^s$ solution} to the unconfined model if
\begin{equation}\label{KP+}
\tfrac{d}{dt}q = p, \quad \tfrac{d}{dt}p = -c^2(q) A q \qtq{in $\H^{s-1}$ sense with}c(q) := \tfrac{1}{1 + \langle q,A q\rangle} .
\end{equation}
In the strongly confined case, a trajectory must satisfy
\begin{equation}\label{KP-}
\begin{gathered}
\tfrac{d}{dt}q = p, \quad \tfrac{d}{dt}p = -c_-^2(q) A q \qtq{in $\H^{s-1}$ sense with}c_-(q) := \tfrac{1}{1 - \langle q,A q\rangle},\\
\text{as well as} \quad  \langle q(t),A q(t)\rangle < 1 \quad\text{for all $t\in I$.}
\end{gathered}
\end{equation}
When $I=\R$, we say that such a solution is \emph{global}. We say that $(q,p)$ is an \emph{$\H^\infty$ solution} if it is an $\H^s$ solution for each $s\geq 1$.
\end{definition}

As discussed in the introduction, throughout the paper we will mainly focus attention to the unconfined model. At the end of each section, we will quickly review any changes needed to cover the strongly confined case.

\begin{lemma}\label{L:GWP}
{\upshape(a)} If $A$ is bounded, then \eqref{KP+} admits a global $\H^1$ solution for each $(q_0,p_0)\in \H^1$; moreover, this solution is unique and conserves the Hamiltonian.

{\upshape(b)} Suppose $(q_0,p_0)\in \H^1$ and $q_0=P_{<N}q_0$ and $p_0= P_{<N}p_0$ for some $N>0$.  Then, \eqref{KP_pm} admits a global $C_t \H^\infty$ solution with this initial data. 
\end{lemma}

\begin{proof}
When $A$ is bounded, \eqref{KP+} corresponds to a Lipschitz vector field on $\H^1$: Adopting the abbreviations $c=c(q)$ and $\tilde c = c(\tilde q)$, we have
\begin{align*}
c^2 Aq - \tilde c^2 A\tilde q = \tfrac{c^2+\tilde c^2}2 A(q-\tilde q)
-\tfrac{c \mskip2mu\tilde c \mskip2mu [c+ \tilde c]}2 \bigl\langle q+\tilde q, A (q-\tilde q)\bigr\rangle A(q+\tilde q)
\end{align*}
and so 
\begin{align*}
\bigl\|(p-\tilde p, -c^2 Aq &+ \tilde c^2 A\tilde q)\bigr\|_{\H^1}^2\\
&\leq \|A\| \|p-\tilde p\|^2 + \|A\| \|A^{\frac12}(q-\tilde q)\|^2\Bigl[ 1+ \tfrac{\|A^{\frac12}(q+\tilde q)\|^2}{(1+\la q, Aq\ra)(1+\la \tilde q, A\tilde q\ra)}\Bigr]^2\\
&\leq 9\|A\|\bigl\|(p-\tilde p, q-\tilde q)\bigr\|_{\H^1}^2 .
\end{align*}
In this way, existence and uniqueness follows from the Picard--Lindel\"of Theorem.  The conservation of the Hamiltonian $H_+$ in \eqref{Hamiltonian} is an elementary computation.

Part (b) follows from applying part (a) to the restriction of $A$ to the (invariant) Hilbert space $\chi_{(-\infty,N)}(A)X$.
\end{proof}

While it is not difficult to show that the solution in case (b) is also unique, we postpone this until Section~\ref{S:5} to avoid repeating ourselves.  There, we will show that \eqref{KP+} is well-posed in $\H^s$ for all $s\in[\frac32,\infty]$.

In the strongly confined case, the vector field is not globally Lipschitz; indeed, it blows up as $\langle q,A q\rangle$ approaches $1$. For true solutions, conservation of the Hamiltonian prevents such an approach; however, such conservation-law arguments do not apply to Picard iterates, they only apply to solutions.  The remedy is standard:  Given $0<H_0<\infty$, we define a new speed function
\begin{equation*}
c(q;H_0) := \tfrac{1}{1 - \langle q,A q\rangle} \qtq{when} \langle q,A q\rangle \leq \tfrac{2H_0}{1+2H_0}; \qtq{otherwise,}  c(q;H_0) := 1+2H_0.
\end{equation*}
This ensures the global Lipschitz property, thereby allowing solutions to the modified equation to be constructed by the Picard--Lindel\"of approach.  Finally, one employs conservation of the Hamiltonian to observe that solutions emanating from initial data satisfying $H_-(q,p)\leq H_0$ never experience the modification made to the speed function.

\section{Conservation laws}\label{S:3}

We start by discussing conservation laws for the unconfined model \eqref{KP+}; to streamline the notation, we will drop the subscript $+$ for the conserved quantities and the speed function. The minor modifications needed for the strongly confined model \eqref{KP-} will be discussed in subsection~\ref{SS:3.1}. 

\begin{proposition} Let $(q,p)$ be an $\H^\infty$ solution to \eqref{KP+}. For any integer $k \geq 0$, 
\begin{equation}\label{Ik}
\begin{aligned}
I_k(q,p) :=  \langle p, A^{k} p\rangle &+ \tfrac{\langle q, A^{k+1} q\rangle}{1 + \langle q, A q\rangle}\\
&+ \sum_{j=1}^{k} \Bigl[ \langle q, A^{j} q\rangle \langle p, A^{k-j+1} p\rangle - \langle p, A^{j} q\rangle \langle p, A^{k-j+1} q\rangle \Bigr]
\end{aligned}
\end{equation}
is conserved by the flow. Moreover, these observables are coercive:
\begin{align}\label{Ik coercive}
\la p, A^{k} p\rangle &+ \tfrac{\langle q, A^{k+1} q\rangle}{1 + \langle q, A q\rangle}\leq I_k(q,p)\lesssim_k\|(q,p)\|_{\H^{k+1}}^2 \bigl[1+\|(q,p)\|_{\H^{k+1}}^2 \bigr].
\end{align}
\end{proposition}

When $k=0$, the sum over $j$ is empty and we get $I_0(q,p)=2H(q,p)$, where $H(q,p)=H_+(q,p)$ is the unconfined Hamiltonian presented in \eqref{Hamiltonian}. When $k=1$, we recover Pohozaev's conserved quantity \eqref{I Poho}.

\begin{proof}
Using \eqref{KP+} and the chain rule, for any $k\geq 0$ we find
\begin{align*}
\tfrac{d}{dt}\la q,A^kq \ra &=  2\la p, A^{k}q\ra,\\
\tfrac{d}{dt}\la p,A^kp \ra &= - \tfrac{2\la p, A^{k+1}q\ra}{(1+\la q, Aq \ra)^2},\\
\tfrac{d}{dt}\la p,A^kq \ra &= -\tfrac{\la q, A^{k+1}q\ra}{(1+\la q, Aq \ra)^2}+\la p, A^k p\ra.
\end{align*}
Using these identities, we compute
\begin{align}\label{243}
\tfrac{d}{dt}\sum_{j=1}^{k} \langle q, A^{j} q\rangle \langle p, A^{k-j+1} p\rangle
&= 2\sum_{j=1}^{k} \langle p, A^{j} q\rangle \langle p, A^{k-j+1} p\rangle \notag\\
&\quad - \tfrac{2}{(1+\la q, Aq \ra)^2}\sum_{j=1}^{k}\langle q, A^{j} q\rangle \langle p, A^{k-j+2} q\rangle
\end{align}
and
\begin{align}\label{244}
\tfrac{d}{dt}\sum_{j=1}^{k} \langle p, A^{j} q\rangle& \langle p, A^{k-j+1} q\rangle=\sum_{j=1}^{k}\la p, A^j p\ra\la p, A^{k-j+1} q\ra+\langle p, A^{j} q\rangle\la p, A^{k-j+1} p\ra\notag\\
&\quad-\tfrac{1}{(1+\la q, Aq \ra)^2}\sum_{j=1}^{k}\la q, A^{j+1}q\ra\la p, A^{k-j+1}q\ra+\langle p, A^{j} q\rangle\la q, A^{k-j+2} q\ra.
\end{align}
Relabeling the summation index $j\mapsto k-j+1$ in the second and fourth summands on RHS\eqref{244}, we deduce
\begin{align*}
\tfrac{d}{dt}\sum_{j=1}^{k} \langle p, A^{j} q\rangle \langle p, A^{k-j+1} q\rangle
&=2\sum_{j=1}^k\la p, A^j p\ra\la p, A^{k-j+1} q\ra\\
&\quad -\tfrac{2}{(1+\la q, Aq \ra)^2}\sum_{j=2}^{k+1}\la q, A^{j}q\ra\la p, A^{k-j+2} q\ra.
\end{align*}
Combining this with \eqref{243}, we arrive at
\begin{align*}
\tfrac{d}{dt} \sum_{j=1}^{k} \Big( \langle q, &\,A^{j} q\rangle \langle p, A^{k-j+1} p\rangle - \langle p, A^{j} q\rangle \langle p, A^{k-j+1} q\rangle \Big)\\
&= \tfrac{-2}{(1+\la q, Aq \ra)^2}\Bigl[\la q, Aq\ra \la p,A^{k+1}q\ra-\la q,A^{k+1}q\ra\la p,Aq\ra\Bigr]\\
&=-\tfrac{d}{dt}\Bigl[\langle p, A^{k} p\rangle + \tfrac{\langle q, A^{k+1} q\rangle}{1 + \langle q, A q\rangle} \Bigr],
\end{align*}
which proves that $I_k(q,p)$ is indeed conserved.

By the Cauchy--Schwarz inequality,
\begin{align*}
\bigl| \langle p, A^{j} q\rangle \langle p, A^{k-j+1} q\rangle\bigr|
&\leq \sqrt{\la p, A^jp\ra \la q, A^j q\ra \la p, A^{k-j+1} p\ra\la q, A^{k-j+1} q\ra}\\
&\leq \tfrac12 \Bigl[\la q, A^jq\ra\la p, A^{k-j+1} p\ra+ \la p, A^jp\ra\la q, A^{k-j+1} q\ra  \Bigr].
\end{align*}
Summing over $1\leq j\leq k$ and performing the change of variables $j\mapsto k-j+1$ in the second summand above, we obtain
$$
\sum_{j=1}^k \bigl| \langle p, A^{j} q\rangle \langle p, A^{k-j+1} q\rangle\bigr|\leq \sum_{j=1}^k\la q, A^jq\ra\la p, A^{k-j+1} p\ra.
$$
This proves the first inequality in \eqref{Ik coercive}. The second inequality in \eqref{Ik coercive} follows from a straightforward application of Cauchy--Schwarz.
\end{proof}

We regard the conserved quantities $I_k(q,p)$ as the polynomial conservation laws for the model \eqref{KP+}. These can be combined into the generating function 
\begin{align}\label{expansion}
I(z;q,p) = \sum_{k\geq 0}(-1)^kz^k I_k(q,p).
\end{align}
This is easily summed explicitly to yield
\begin{equation}\label{I+}
\begin{aligned}
I(z;q,p)&=\langle p,\tfrac{1}{1+zA}p\rangle+c(q)\langle q,\tfrac{A}{1+zA}q\rangle\\
&\qquad - z\bigl[\langle q,\tfrac{A}{1+ zA}q\rangle\langle p,\tfrac{A}{1+ zA}p\rangle-\langle p,\tfrac{A}{1+ zA}q\rangle^2\bigr].
\end{aligned}
\end{equation}
Conversely, one can formally recover \eqref{Ik} from \eqref{I+} by expanding $(1+zA)^{-1} = \sum_{k\geq 0} (-1)^kz^kA^k$ and collecting terms with equal powers of $z$.

One virtue of the generating function is that it is well defined throughout $\H^1$; by contrast, the polynomial conservation law $I_k(q,p)$ requires $(q,p)\in \H^{k+1}$.

Our next result demonstrates that the generating function is conserved by the flow \eqref{KP+} and that $\{I(z;q,p)\}_{z\geq 0}$ defines a family of Poisson commuting observables.  

\begin{proposition}\label{P:Poisson} For any $z,w \geq 0$, we have
\begin{equation}\label{Poisson commute}
\{I(w;q,p),I(z;q,p)\}=0 \qtq{on} \H^\infty.
\end{equation}
In particular, as $I(0;q,p) =2H(q,p)$, the quantity $I(z)$ is conserved for all $\H^\infty$ solutions to \eqref{KP+}. More generally,
\begin{align}\label{1153}
\{I(z;q,p),I_k(q,p)\}=0 \qtq{and} \{I_k(q,p),I_\ell(q,p)\}=0
\end{align}
for all $z\geq0$ and integers $k,\ell\geq 0$.
\end{proposition}

\begin{proof}
We decompose $I (z)=I^{(1)}(z)+I^{(2)}(z)+I^{(3)}(z)$, where
\begin{align*}
I^{(1)}(z)&:=\langle p,\tfrac{1}{1+zA}p\rangle, \quad  I^{(2)}(z) := c(q)\langle q,\tfrac{A}{1+zA}q\rangle \qtq{and} \\
I^{(3)}(z)& := -z\Bigl[\langle q,\tfrac{A}{1+zA}q\rangle\langle p,\tfrac{A}{1+zA}p\rangle-\langle p,\tfrac{A}{1+zA}q\rangle^2\Bigr].
\end{align*}
We have 
\begin{align}
\tfrac{\partial}{\partial q} I^{(1)} (z)&= 0 \qtq{and} \tfrac{\partial}{\partial p} I^{(1)} (z)= 2 \tfrac{1}{1+zA}p \label{eq:partial 1} \\
\tfrac{\partial}{\partial q} I^{(2)} (z)&= 2 c(q)  \tfrac{A}{1+zA}q - 2 c(q)^2\langle q, \tfrac{A}{1+zA}q\rangle Aq \qtq{and} \tfrac{\partial}{\partial p} I^{(2)}(z) = 0\label{eq:partial 2}\\
\tfrac{\partial}{\partial q} I^{(3)} (z)&= -2z\langle p, \tfrac{A}{1+zA}p\rangle  \tfrac{A}{1+zA}q +2z \langle p, \tfrac{A}{1+zA}q\rangle  \tfrac{A}{1+zA}p \label{eq:partial 3}\\
\tfrac{\partial}{\partial p} I^{(3)}(z) &= -2z\langle q, \tfrac{A}{1+zA}q\rangle  \tfrac{A}{1+zA}p +2z \langle p, \tfrac{A}{1+zA}q\rangle  \tfrac{A}{1+zA}q. \label{eq:partial 4}
\end{align}

By \eqref{eq:partial 1} and \eqref{eq:partial 2}, it is easy to see that
\begin{align}\label{eq:poisson_commutativity 1}
\{I^{(1)}(w),I ^{(1)}(z) \}=0 \qtq{and} \{I^{(2)}(w),I ^{(2)}(z) \}=0.
\end{align}

Using \eqref{eq:partial 1} and \eqref{eq:partial 3}, we find
\begin{align*}
&\{I^{(1)}(w),I^{(3)}(z)\} + \{I^{(3)}(w),I^{(1)}(z)\}\\ 
&=-\bigl\langle \tfrac{\partial}{\partial p} I^{(1)}(w),\tfrac{\partial}{\partial q} I^{(3)}(z) \bigr\rangle
	+ \bigl\langle \tfrac{\partial}{\partial q} I^{(3)}(w),\tfrac{\partial}{\partial p} I^{(1)}(z) \bigr\rangle\\
&= 4\la p,\tfrac{A}{(1+zA)(1+wA)}q\ra\la p, \tfrac{zA}{1+z A}p\ra- 4\la p,\tfrac{A}{(1+zA)(1+wA)}p\ra\la p, \tfrac{zA}{1+zA}q\ra\\
&\quad-4\la p,\tfrac{A}{(1+zA)(1+wA)}q\ra\la p, \tfrac{wA}{1+wA}p\ra+4\la p,\tfrac{A}{(1+zA)(1+wA)}p\ra\la p, \tfrac{wA}{1+wA}q\ra.
\end{align*}
Employing the identity
\begin{equation}\label{id}
 \tfrac{zA}{1+z A}- \tfrac{wA}{1+w A} = (z-w) \tfrac{A}{(1+z A)(1+wA)},
\end{equation}
this simplifies to 
\begin{align}\label{eq:poisson_commutativity 2}
\{I^{(1)}(w),I^{(3)}(z)\} &+ \{I^{(3)}(w),I^{(1)}(z)\} \\ 
&=4(z-w)\la p,\tfrac{A}{(1+zA)(1+wA)}q\ra\la p, \tfrac{A}{(1+z A)(1+wA)} p\ra\notag\\
&\quad-4(z-w)\la p,\tfrac{A}{(1+zA)(1+wA)}p\ra\la p, \tfrac{A}{(1+z A)(1+wA)} q\ra\notag\\
&=0.\notag
\end{align}

Next we show that
\begin{align}\label{eq:poisson_commutativity 4}
\{I^{(3)}(w),I ^{(3)}(z) \}=0.
\end{align}
A straightforward computation employing \eqref{eq:partial 3} and \eqref{eq:partial 4} yields
\begin{align*}
\{I^{(3)}&(w),I ^{(3)}(z) \}\\
&=4zw \la p, \tfrac{A^2}{(1+zA)(1+wA)}q\ra\Bigl[  \la q, \tfrac{A}{1+zA}q\ra\la p, \tfrac{A}{1+wA}p\ra -  \la q, \tfrac{A}{1+wA}q\ra\la p, \tfrac{A}{1+zA}p\ra\Bigr]\\
&\quad-4zw\la p, \tfrac{A^2}{(1+zA)(1+wA)}p\ra\Bigl[  \la q, \tfrac{A}{1+zA}q\ra\la p, \tfrac{A}{1+wA}q\ra -  \la q, \tfrac{A}{1+wA}q\ra\la p, \tfrac{A}{1+zA}q\ra\Bigr]\\
&\quad-4zw\la q, \tfrac{A^2}{(1+zA)(1+wA)}q\ra\Bigl[  \la p, \tfrac{A}{1+zA}q\ra\la p, \tfrac{A}{1+wA}p\ra -  \la p, \tfrac{A}{1+wA}q\ra\la p, \tfrac{A}{1+zA}p\ra\Bigr].
\end{align*}
To deduce the claim \eqref{eq:poisson_commutativity 4}, we apply the following identity in each of the square brackets above:
\begin{align}\label{id 2}
 \la f, \tfrac{A}{1+zA}g\ra\la \phi, \tfrac{A}{1+wA}\psi\ra &-  \la f, \tfrac{A}{1+wA}g\ra\la \phi, \tfrac{A}{1+zA}\psi\ra\notag\\
 &=(z-w) \la f, \tfrac{A}{(1+zA)(1+wA)}g\ra \la \phi, \tfrac{A^2}{(1+zA)(1+wA)}\psi\ra\\
 &\quad - (z-w) \la f, \tfrac{A^2}{(1+zA)(1+wA)}g\ra \la \phi, \tfrac{A}{(1+zA)(1+wA)}\psi\ra .\notag
\end{align}
More precisely, we take $f=g=q$ and $\phi=\psi=p$ in the first term, $f=g=\psi=q$ and $\phi=p$ in the second, and $f=\phi=\psi=p$ and $g=q$ in the third.

In view of \eqref{eq:poisson_commutativity 1}, \eqref{eq:poisson_commutativity 2}, and \eqref{eq:poisson_commutativity 4}, claim \eqref{Poisson commute} will follow once we show
\begin{align}\label{eq:poisson_commutativity 3}
\{I^{(1)}(w)+I^{(3)}(w),I ^{(2)}(z) \} +	\{I^{(2)}(w),I ^{(1)}(z)+I ^{(3)}(z) \} =0.
\end{align}
To this end, we use \eqref{eq:partial 2} and \eqref{eq:partial 4} to compute
\begin{align*}
&\{I^{(2)}(w),I^{(3)}(z) \} \\
&= \bigl\langle \tfrac{\partial}{\partial q} I^{(2)}(w),\tfrac{\partial}{\partial p} I^{(3)}(z) \bigr\rangle\\
&=-4c(q)\la q, \tfrac{zA}{1+zA}q\ra \la p, \tfrac{A^2}{(1+z A)(1+wA)} q\ra +4c(q)\la p, \tfrac{zA}{1+zA}q\ra \la q, \tfrac{A^2}{(1+z A)(1+wA)} q\ra\\
&\quad +4 c(q)^2 \la q, \tfrac{A}{1+wA}q\ra\la q, \tfrac{A}{1+zA}q\ra\la p, \tfrac{zA^2}{1+zA}q\ra \\
&\quad-4 c(q)^2 \la q, \tfrac{A}{1+wA}q\ra\la p, \tfrac{A}{1+zA}q\ra\la q, \tfrac{zA^2}{1+zA}q\ra\\
&=-4c(q)\la q, \tfrac{zA}{1+zA}q\ra \la p, \tfrac{A^2}{(1+z A)(1+wA)} q\ra +4c(q)\la p, \tfrac{zA}{1+zA}q\ra \la q, \tfrac{A^2}{(1+z A)(1+wA)} q\ra\\
&\quad +4 c(q)^2 \la q, \tfrac{A}{1+wA}q\ra\la q, \tfrac{A}{1+zA}q\ra\la p, \tfrac{zA^2}{1+zA}q\ra \\
&\quad-4 c(q)^2 \la q, \tfrac{A}{1+wA}q\ra\la p, \tfrac{A}{1+zA}q\ra\Bigl[ \la q, Aq\ra-  \la q, \tfrac{A}{1+zA}q\ra\Bigr] .
\end{align*}
From this (both as written and after interchanging the $z$ and $w$ variables) and the identity \eqref{id}, we obtain
\begin{align}\label{441}
\{I^{(2)}(w),I ^{(3)}(z) \} & +\{I^{(3)}(w),I ^{(2)}(z) \} \notag\\
&=-4(z-w)c(q)\la q, \tfrac{A}{(1+z A)(1+wA)}q\ra \la p, \tfrac{A^2}{(1+z A)(1+wA)} q\ra\notag\\
&\quad+4(z-w)c(q)\la p, \tfrac{A}{(1+z A)(1+wA)}q\ra \la q, \tfrac{A^2}{(1+z A)(1+wA)} q\ra\notag\\
&\quad +4 (z-w) c(q)^2 \la q, \tfrac{A}{1+wA}q\ra\la q, \tfrac{A}{1+zA}q\ra\la p, \tfrac{A^2}{(1+zA)(1+wA)}q\ra\notag\\
&\quad- 4(z-w) c(q)^2 \la q, \tfrac{A}{1+wA}q\ra\la q, \tfrac{A}{1+zA}q\ra\la p, \tfrac{A^2}{(1+zA)(1+wA)}q\ra \notag\\
&\quad + 4c(q)^2 \la q,Aq\ra \Bigl[ \la q, \tfrac{A}{1+zA}q\ra \la p, \tfrac{A}{1+wA}q\ra - 
	\la q, \tfrac{A}{1+wA}q\ra \la p, \tfrac{A}{1+zA}q\ra  \Bigr] \notag\\
&=-4(z-w)c(q)\la q, \tfrac{A}{(1+z A)(1+wA)}q\ra \la p, \tfrac{A^2}{(1+z A)(1+wA)} q\ra\\
&\quad+4(z-w)c(q)\la p, \tfrac{A}{(1+z A)(1+wA)}q\ra \la q, \tfrac{A^2}{(1+z A)(1+wA)} q\ra\notag\\
&\quad + 4c(q)^2 \la q,Aq\ra \Bigl[ \la q, \tfrac{A}{1+zA}q\ra \la p, \tfrac{A}{1+wA}q\ra - 
	\la q, \tfrac{A}{1+wA}q\ra \la p, \tfrac{A}{1+zA}q\ra  \Bigr]. \notag
\end{align}

Using \eqref{eq:partial 1} and \eqref{eq:partial 2}, we compute
\begin{align}\label{440}
\{I^{(1)}(w),&\, I ^{(2)}(z) \}+ \{I^{(2)}(w),I ^{(1)}(z) \} \notag\\
&= -\bigl\langle \tfrac{\partial}{\partial p} I^{(1)}(w),\tfrac{\partial}{\partial q} I^{(2)}(z) \bigr\rangle
	+ \bigl\langle \tfrac{\partial}{\partial q} I^{(2)}(w),\tfrac{\partial}{\partial p} I^{(1)}(z) \bigr\rangle\\
&=4c(q)^2 \Bigl[ \la q,\tfrac{A}{1+zA}q\ra \la p, \tfrac{A}{1+wA}q\ra  -  \la q, \tfrac{A}{1+wA}q\ra \la p,\tfrac{A}{1+zA}q\ra \Bigr], \notag
\end{align}
which we then combine with \eqref{441} to obtain
\begin{align*}
\{I^{(1)}&(w)+I^{(3)}(w),I ^{(2)}(z) \} +	\{I^{(2)}(w),I ^{(1)}(z)+I ^{(3)}(z) \} \\
&=-4(z-w)c(q)\la q, \tfrac{A}{(1+z A)(1+wA)}q\ra \la p, \tfrac{A^2}{(1+z A)(1+wA)} q\ra\\
&\quad+4(z-w)c(q)\la p, \tfrac{A}{(1+z A)(1+wA)}q\ra \la q, \tfrac{A^2}{(1+z A)(1+wA)} q\ra\notag\\
&\quad + 4c(q) \Bigl[ \la q, \tfrac{A}{1+zA}q\ra \la p, \tfrac{A}{1+wA}q\ra - 
	\la q, \tfrac{A}{1+wA}q\ra \la p, \tfrac{A}{1+zA}q\ra  \Bigr]. \notag
\end{align*}
Lastly, we apply the identity \eqref{id 2} with $f=g=\psi= q$ and $\phi=p$ to deduce the result \eqref{eq:poisson_commutativity 3}.  The claims in \eqref{1153} follow by examining Taylor coefficients about $w=0$ and $z=0$.
\end{proof}

\subsection{Conservation laws for the strongly confined model}\label{SS:3.1}

In this subsection, we discuss conservation laws for the strongly confined model \eqref{KP-}.

The analogue of the generating function $I(z;q,p)$ is
\begin{equation}\label{eq:definition_Iz_minus}
\begin{aligned}
I_-(z;q,p)&=\langle p,\tfrac{1}{1+zA}p\rangle+c_-(q)\langle q,\tfrac{A}{1+zA}q\rangle\\
&\quad +z\bigl[\langle q,\tfrac{A}{1+ zA}q\rangle\langle p,\tfrac{A}{1+ zA}p\rangle-\langle p,\tfrac{A}{1+ zA}q\rangle^2\bigr],
\end{aligned}
\end{equation}
which is well defined whenever $(q,p)\in \H^1$ and $z\geq 0$. Expanding about $z=0$, yields the polynomial conservation laws
\begin{align*}
I_k^-(q,p) :=  \langle p, A^{k} p\rangle &+ \tfrac{\langle q, A^{k+1} q\rangle}{1 - \langle q, A q\rangle}\\
&- \sum_{j=1}^{k} \Bigl[ \langle q, A^{j} q\rangle \langle p, A^{k-j+1} p\rangle - \langle p, A^{j} q\rangle \langle p, A^{k-j+1} q\rangle \Bigr].
\end{align*}

It is easy to adapt the proof of Proposition~\ref{P:Poisson} to verify that $\{I_-(z;q, p)\}_{z\geq 0}$ forms a Poisson commuting family of observables.  Indeed, mimicking the proof of Proposition~\ref{P:Poisson} with
\begin{align*}
I_-^{(1)}(z):= I^{(1)}(z), \quad I_-^{(2)}(z):= c_-(q)\langle q,\tfrac{A}{1+zA}q\ra, \quad I_-^{(3)}(z):=- I^{(3)}(z),
\end{align*}
and 
\begin{align*}
\tfrac{\partial}{\partial q} I_-^{(2)} (z)&= 2 c_-(q)  \tfrac{A}{1+zA}q + 2 c_-(q)^2\langle q, \tfrac{A}{1+zA}q\rangle Aq  \qtq{and} \tfrac{\partial}{\partial p} I_-^{(2)}(z) = 0
\end{align*}
in place of \eqref{eq:partial 2}, we obtain
\begin{equation*}
\{I_-(w;q,p),I_-(z;q,p)\}=0.
\end{equation*}

\section{Uniform bounds and equicontinuity}\label{S:4}

The goal of this section is to prove a priori bounds and equicontinuity of orbits under the flow \eqref{KP+}. We do this by further elucidating properties of the generating function $I(z;q,p)$.  As before, the lack of subscripts indicates the unconfined model. The modifications needed to obtain analogous results for the strongly confined model will be discussed in subsection~\ref{SS:4.1}.

\begin{definition}\label{property}
Given $\Omega\subset \H^\infty$ and $T>0$, we define the set of orbits emanating from $\Omega$ for times $|t|\leq T$ via
\begin{align*}
\Omega_T&:= \bigl\{ (q(t),p(t))\,|\, (q,p)\text{ is an $\H^\infty$ solution to \eqref{KP+} with $(q(0), p(0))\in \Omega$}\\
&\qquad\qquad\qquad\qquad\qquad   \text{and $t\in I\cap[-T,T]$}\bigr\}.
\end{align*}
\end{definition}

At this moment, we have not yet proved that $\H^\infty$ solutions to \eqref{KP+} exist and are global, nor that they are uniquely determined by the initial data. This explains the slightly cumbersome formulation of this definition compared to \eqref{T:1.2a}.

Our first objective is proving a priori bounds and equicontinuity at the critical regularity:

\begin{theorem}[$\H^1$ bounds and equicontinuity]\label{T:equicontinuity_H1} Suppose $\Omega \subset \H^\infty$ is bounded and equicontinuous in $\H^1$. For each $T>0$, $\Omega_T$ is bounded and equicontinuous in $\H^1$ with
\begin{align}\label{H1 bdd}
\sup_{(q,p)\in \Omega_T} \,\|p\| &\leq C(\Omega) \qtq{and} \sup_{(q,p)\in \Omega_T} \, \| A^{\frac12} q\| \leq C(\Omega)\la T\ra.
\end{align}
\end{theorem}

\begin{proof}
The first bound in \eqref{H1 bdd} follows from the conservation of the unconfined Hamiltonian in \eqref{Hamiltonian}:
\begin{align}\label{p bdd}
\|p(t)\|^2\leq 2H(q,p)\leq \|p(0)\|^2+1.
\end{align}

To prove the second bound in \eqref{H1 bdd}, we will show that there exists $z_\Omega >0$ such that for any $0<z\leq z_{\Omega}$,
\begin{equation}\label{eq:boundqAq_equicontinuity}
\la q(t), Aq(t) \ra \leq 1 + 4\langle q(0), \tfrac{A}{1+zA}q(0)\rangle + \tfrac{8}{z}H(q(0),p(0)) t^2,
\end{equation}
uniformly for $(q(0),p(0))\in \Omega$ and all times of existence $t\in [-T,T]$. 

Using $\sqrt{\frac{A}{1+zA}} \leq \frac{1}{\sqrt{z}}$ in the sense of quadratic forms and \eqref{p bdd}, we compute
\begin{align*}
\tfrac{d}{dt}\langle q, \tfrac{A}{1+zA}q\rangle = 2 \langle p,\tfrac{A}{1+zA} q\rangle
&\leq \tfrac{2}{\sqrt{z}} \la p, p \ra^{\frac12}\, \langle q,\tfrac{A}{1+zA}q\rangle^{\frac12}\\
&\leq \tfrac{2\sqrt2}{\sqrt z} H(q(0),p(0))^{\frac12}\langle q,\tfrac{A}{1+zA}q\rangle^{\frac12}.
\end{align*}
Thus, for all $z>0$ and times of existence $t\in [-T,T]$ we obtain
\begin{equation*}
\langle q(t), \tfrac{A}{1+zA}q(t)\rangle^{\frac12} \leq   \langle q(0), \tfrac{A}{1+zA}q(0)\rangle^{\frac12} + \tfrac{\sqrt2}{\sqrt{z}}H(q(0),p(0))^{\frac12} |t|
\end{equation*}
and so
\begin{equation}\label{eq:boundqAq}
\langle q(t), \tfrac{A}{1+zA}q(t)\rangle \leq   2\langle q(0), \tfrac{A}{1+zA}q(0)\rangle + \tfrac{4}{z}H(q(0),p(0)) t^2.
\end{equation}

To connect \eqref{eq:boundqAq} to \eqref{eq:boundqAq_equicontinuity}, we will employ the conservation of
\begin{equation}\label{I diff}
\begin{aligned}
I(0;q,p)-I(z;q,p) &= \langle p, \tfrac{zA}{1+zA}p\rangle + c(q) \langle q, \tfrac{zA^2}{1+zA}q\rangle \\
&\quad+  \Bigl[ \langle q, \tfrac{A}{1+zA}q\rangle\langle p, \tfrac{zA}{1+zA}p\rangle -\langle p, \sqrt{\tfrac{zA}{1+zA}}\sqrt{\tfrac{A}{1+zA}}q\rangle^2  \Bigr].
\end{aligned}
\end{equation}
Importantly, this quantity is coercive. Indeed, by the Cauchy--Schwarz inequality,
\begin{align*}
 \langle q, \tfrac{A}{1+zA}q\rangle\langle p, \tfrac{zA}{1+zA}p\rangle -\langle p, \sqrt{\tfrac{zA}{1+zA}}\sqrt{\tfrac{A}{1+zA}}q\rangle^2 \geq 0
\end{align*}
and consequently,
\begin{equation}\label{CS}
\begin{aligned}
\langle p, \tfrac{zA}{1+zA}p\rangle + c(q) \langle q, \tfrac{zA^2}{1+zA}q\rangle
&\leq I(0;q,p)-I(z;q,p)\\
&\leq \bigl[1+\|A^{\frac12}q\|^2\bigr]\Bigl\| \sqrt{\tfrac{zA}{1+zA}}p\Bigr\|^2+ \Bigl\|\sqrt{\tfrac{zA}{1+zA}}A^{\frac12}q\Bigr\|^2.
\end{aligned}
\end{equation}

As $\Omega$ is bounded and equicontinuous in $\H^1$, we have
\begin{align*}
\lim_{z\to 0}\, \sup_{(q,p)\in \Omega}\, \Bigl\| \sqrt{\tfrac{zA}{1+zA}}p\Bigr\|^2+ \Bigl\|\sqrt{\tfrac{zA}{1+zA}}A^{\frac12}q\Bigr\|^2=0.
\end{align*}
Combining this with \eqref{CS}, we see that we may choose $z_{\Omega}>0$ such that
\begin{align}\label{I cont}
I(0;q,p)-I(z;q,p) \leq \tfrac12 \qtq{for all} 0<z\leq z_\Omega \qtq{and} (q,p)\in \Omega.
\end{align}

Using the conservation of \eqref{I diff} under the flow \eqref{KP+} together with \eqref{CS} and \eqref{I cont}, we find
\begin{equation*}
\sup_{(q,p)\in \Omega_T}\, c(q) \langle q, \tfrac{zA^2}{1+zA}q\rangle \leq \tfrac12  \qtq{for all} 0<z\leq z_\Omega.
\end{equation*}
Writing $\frac{zA^2}{1+zA} = A- \tfrac{A}{1+zA}$, we arrive at
\begin{equation*}
\la q, Aq\ra \leq 1 + 2\la q, \tfrac{A}{1+zA}q\ra \qtq{for all} 0<z\leq z_\Omega \qtq{and} (q,p)\in \Omega_T.
\end{equation*}
Claim \eqref{eq:boundqAq_equicontinuity} follows from this and \eqref{eq:boundqAq}.

We now turn to the propagation of equicontinuity under the flow \eqref{KP+}. Using the conservation of \eqref{I diff} and \eqref{CS}, we obtain
\begin{align}\label{659}
\langle p(t),\tfrac{zA}{1+zA}p(t)\rangle &+ c(q(t)) \langle q(t),\tfrac{zA^2}{1+zA}q(t)\rangle\notag\\
&\leq  \bigl[1+ \|A^{\frac12}q(0)\|^2\bigr]\langle p(0),\tfrac{zA}{1+zA}p(0)\rangle + \langle q(0), \tfrac{zA^2}{1+zA} q(0)\rangle,
\end{align}
uniformly for $(q(0),p(0))\in \Omega$ and times of existence $t\in[-T,T]$. Setting $z = \kappa^{-2}$ for $\kappa\geq 1$, we may bound
\begin{align*}
\langle p(t),P_{\geq  \kappa}\, p(t)\rangle &+ c(q(t))\langle q(t),P_{\geq  \kappa}\, Aq(t)\rangle\\
&\lesssim\langle p(t),\tfrac{A}{\kappa^2+A}p(t)\rangle + c(q(t)) \langle q(t),\tfrac{A^2}{\kappa^2+A}q(t)\rangle\\
&\lesssim \bigl[1+ \|A^{\frac12}q(0)\|^2\bigr]\Bigl[\tfrac1\kappa\la p(0),p(0)\ra + \langle p(0),P_{\geq \sqrt \kappa}\, p(0)\rangle\Bigr] \\
&\quad+  \tfrac1\kappa\langle q(0),A q(0)\rangle+ \langle q(0),P_{\geq \sqrt \kappa}\, Aq(0)\rangle,
\end{align*}
uniformly for $(q(0),p(0))\in \Omega$ and times of existence $t\in[-T,T]$.  The equicontinuity of $\Omega_T$ now follows from the boundedness and equicontinuity of $\Omega$ in $\H^1$ together with \eqref{H1 bdd}.
\end{proof}

Our next goal is to prove a priori bounds and equicontinuity of orbits in higher regularity spaces.

\begin{theorem}[$\H^s$ bounds and equicontinuity]\label{T:equicontinuity_Hs} For each $s>1$, we have the following universal bound for $\H^\infty$ solutions to \eqref{KP+}:
\begin{equation}\label{Hs bdd}
\begin{aligned}
\|p(t)\|_{H^{s-1}} &\leq C\bigl(\|(q(0),p(0))\|_{\H^s}\bigr)\\
\|A^\frac12 q(t)\|_{H^{s-1}} &\leq C\bigl(\|(q(0),p(0))\|_{\H^s}\bigr) \la t\ra.
\end{aligned}
\end{equation}
Moreover, if $\Omega \subset \H^\infty$ is bounded and equicontinuous in $\H^s$, then so too is the set of orbits $\Omega_T$ for any $T>0$.
\end{theorem}

\begin{proof}
As $\H^s$ balls are $\H^1$ equicontinuous, Theorem~\ref{T:equicontinuity_H1} guarantees
\begin{equation}\label{H1 bdd'}
 \|p(t)\| \leq C\bigl(\|(q(0),p(0))\|_{\H^s}\bigr) \qtq{and}
\|A^\frac12 q(t)\|\leq C\bigl(\|(q(0),p(0))\|_{\H^s}\bigr) \la t\ra.
\end{equation}

We first present the details for regularities $1<s<2$. In this case, we further build on \eqref{659}, which was derived by analyzing the difference $I(0;q,p) - I(z;q,p)$. Setting $z=\kappa^{-2}$ in \eqref{659} we obtain
\begin{align*}
\langle p(t),\tfrac{A}{\kappa^2+A}p(t)\rangle &+ c(q(t)) \langle q(t),\tfrac{A^2}{\kappa^2+A}q(t)\rangle\\
&\lesssim \bigl[1+ \|A^{\frac12}q(0)\|^2\bigr]\langle p(0),\tfrac{A}{\kappa^2+A}p(0)\rangle + \langle q(0), \tfrac{A^2}{\kappa^2+A} q(0)\rangle.
\end{align*}
Integrating this against $\kappa^{2s-3}$ and using 
\begin{equation*}
\int_{\kappa_0}^{\infty} \frac{\kappa^{2s-3}}{\kappa^2 + \lambda^2} \,d\kappa\approx_{s} (\kappa_0^2+\lambda^2)^{s-2},
\end{equation*}
we obtain
\begin{align}\label{eq:uniform_Hs_control}
\langle p(t), (\kappa_0^2&+A)^{s-2} Ap(t)\rangle + c(q(t)) \langle q(t), (\kappa_0^2+A)^{s-2}  A^2q(t)\rangle\notag\\
&\hspace{-3pt}\lesssim_s \bigl[1+ \|A^{\frac12}q(0)\|^2\bigr]\langle p(0), (\kappa_0^2+A)^{s-2} Ap(0)\rangle +  \langle q(0), (\kappa_0^2+A)^{s-2} A^2q(0)\rangle.
\end{align}
In particular, taking $\kappa_0=0$ in \eqref{eq:uniform_Hs_control} we get
\begin{align*}
\langle p(t), A^{s-1} p(t)\rangle &+ c(q(t)) \langle q(t),  A^sq(t)\rangle\\
&\quad\lesssim_s \bigl[1+ \|A^{\frac12}q(0)\|^2\bigr]\langle p(0),A^{s-1} p(0)\rangle + \langle q(0), A^sq(0)\rangle \\
&\quad \lesssim_s \|(q(0),p(0))\|_{ \H^s}^2\bigl[1+ \|(q(0),p(0))\|_{\H^1}^2\bigr]. 
\end{align*}
Claim \eqref{Hs bdd} follows from this and \eqref{H1 bdd'}.

Equicontinuity of $\Omega_T$ in $\H^s$ also follows from \eqref{eq:uniform_Hs_control}.  Indeed, inserting
\begin{align*}
\langle p(t),P_{\geq  \kappa_0}\, &A^{s-1}p(t)\rangle + c(q(t))\langle q(t),P_{\geq  \kappa_0}\, A^s q(t)\rangle\\
&\lesssim \langle p(t), (\kappa_0^2+A)^{s-2} Ap(t)\rangle + c(q(t)) \langle q(t), (\kappa_0^2+A)^{s-2}  A^2q(t)\rangle
\end{align*}
and
\begin{align*}
\la p(0), (\kappa_0^2+A)^{s-2}A p(0)\ra &\lesssim \kappa_0^{s-2}\la p(0), A^{s-1}p(0)\ra + \la p(0), P_{\geq \sqrt \kappa_0} A^{s-1}p(0)\ra,\\
\la q(0), (\kappa_0^2+A)^{s-2}A^2 q(0)\ra &\lesssim  \kappa_0^{s-2}\la q(0), A^{s}q(0)\ra + \la q(0), P_{\geq \sqrt \kappa_0} A^{s}q(0)\ra
\end{align*}
into \eqref{eq:uniform_Hs_control} and using \eqref{H1 bdd'} and the $ \H^s$ boundedness and equicontinuity of $\Omega$, we see that $\Omega_T$ is also $ \H^s$ equicontinuous.

We now turn to the case of higher regularities $s\geq 2$. If $s=\ell+1$ with $\ell\in \N$, then \eqref{Hs bdd} follows from \eqref{Ik coercive}, the conservation of $I_{\ell}(q,p)$ under the flow \eqref{KP+}, and \eqref{H1 bdd'}. To prove $\H^s$ equicontinuity of $\Omega_T$ when $s=\ell+1$, as well as a priori bounds and equicontinuity when $\ell+1<s<\ell+2$, we use a higher order Taylor expansion of $I(z;q,p)$.  Specifically, for an integer $1\leq \ell\leq s-1$ we introduce the quantity
\begin{equation}\label{R}
R_\ell(z;q,p) = (-1)^{\ell+1} z^{-\ell-1}\biggl[I(z;q,p) -\sum_{j=0}^{\ell} \frac{I^{(j)}(0;q,p)z^j}{j!}\biggr].
\end{equation}
By \eqref{Ik coercive}, the quantity $R_\ell(z;q,p)$ is finite for $(q,p)\in \H^{\ell+1}\subseteq \H^s$.  As $I(z;q,p)$ is conserved by the flow \eqref{KP+}, so is $R_\ell(z;q,p)$.

Using either the definition \eqref{I+} of $I(z;q,p)$ or the expansion \eqref{expansion}, it is easy to verify that
$$
I^{(j)}(0;q,p)= (-1)^j j! \,I_j(q,p),
$$
where $I_j(q,p)$ are  as defined in \eqref{Ik}. In this way, after considerable rearrangement we arrive at 
\begin{align*}
R_\ell(z;q,p )&=\la p, \tfrac{A^{\ell+1}}{1+zA}p\ra + c(q)\la q, \tfrac{A^{\ell+2}}{1+zA}q\ra\\ 
&\quad+\tfrac12\Bigl[\la q, \tfrac{A}{1+zA}q\ra\la p, \tfrac{A^{\ell+1}}{1+zA}p\ra +\la p,  \tfrac{A}{1+zA}p\ra \la q,\tfrac{A^{\ell+1}}{1+zA}q\ra \\
&\quad\qquad\qquad+\sum_{j=0}^{\ell-1}\la q, A^{j+1} q\ra\la p,\tfrac{A^{\ell+1-j}}{1+zA}p\ra+\la p, A^{j+1} p\ra \la q,\tfrac{A^{\ell+1-j}}{1+zA}q\ra \Bigr]\\
&\quad - \Bigl[\la p, \tfrac{A^{\ell+1}}{1+zA}q\ra\la p, \tfrac{A}{1+zA}q\ra+\sum_{j=0}^{\ell-1}\la p, A^{j+1} q\ra \la p,\tfrac{A^{\ell+1-j}}{1+zA}q\ra\Bigr].
\end{align*}
An application of the Cauchy--Schwarz inequality then yields 
\begin{align}\label{eq:uniform_control_H^k}
\la p, \tfrac{A^{\ell+1}}{1+zA}p\ra + c(q)\la q, \tfrac{A^{\ell+2}}{1+zA}q\ra
\leq R_\ell(z;q,p) &\lesssim\la p, \tfrac{A^{\ell+1}}{1+zA}p\ra + \la q,\tfrac{A^{\ell+2}}{1+zA}q\ra \notag\\
&\quad +\|A^{\frac12}q\|_{ H^{\ell}}^2\bigl[\la p,\tfrac{A}{1+zA}p\ra + \la p,\tfrac{A^{\ell+1}}{1+zA}p\ra\bigr]\\
&\quad +\|p\|_{ H^\ell}^2\bigl[\la q,\tfrac{A^{2}}{1+zA}q\ra + \la q,\tfrac{A^{\ell+1}}{1+zA}q\ra\bigr].\notag
\end{align}
Taking $z=\kappa^{-2}$ in \eqref{eq:uniform_control_H^k} and using the conservation of $R_\ell(z;q,p)$, we find
\begin{align}\label{613}
\la p(t), \tfrac{A^{\ell+1}}{\kappa^2+A}&p(t)\ra + c(q(t))\la q(t), \tfrac{A^{\ell+2}}{\kappa^2+A}q(t)\ra\notag\\
&\quad\lesssim\la p(0), \tfrac{A^{\ell+1}}{\kappa^2+A}p(0)\ra + \la q(0),\tfrac{A^{\ell+2}}{\kappa^2+A}q(0)\ra\notag\\
&\qquad +\|(q(0),p(0))\|_{\H^{\ell+1}}^2\Bigl[\la p(0),\tfrac{A}{\kappa^2+A}p(0)\ra + \la p(0),\tfrac{A^{\ell+1}}{\kappa^2+A}p(0)\ra\\
&\qquad\qquad\qquad \qquad\qquad\qquad \qquad+ \la q(0),\tfrac{A^{2}}{\kappa^2+A}q(0)\ra + \la q,\tfrac{A^{\ell+1}}{\kappa^2+A}q(0)\ra\Bigr].\notag
\end{align}
Using
\begin{equation}\label{eq:comparison}
P_{\geq \kappa} \lesssim \tfrac{A}{\kappa^2 + A} \lesssim \tfrac1\kappa \,I + P_{\geq \sqrt \kappa}
\end{equation}
together with the boundedness and equicontinuity of $\Omega$ in $\H^{\ell+1}$, we conclude that $\Omega_T$ is equicontinuous in $\H^{\ell+1}$.

For $\ell+1<s<\ell+2$, we integrate \eqref{613} against $\kappa^{2s-2\ell-3}$ and employ
\begin{align*}
\int_{k_0}^{\infty} \frac{\kappa^{2s-2\ell-3}}{\kappa^2 + \lambda^2} \,d\kappa\sim_{s,\ell} (\kappa_0^2+\lambda^2)^{s-(\ell+2)}.
\end{align*}
As in the case $1<s<2$ treated above, \eqref{Hs bdd} then follows from taking $\kappa_0=0$ and invoking \eqref{H1 bdd'}. The equicontinuity of $\Omega_T$ in $\H^s$ follows from \eqref{eq:comparison}, the boundedness and equicontinuity of $\Omega$ in $\H^s$, and \eqref{H1 bdd'}.
\end{proof}

\subsection{The strongly confined model}\label{SS:4.1}

This subsection is devoted to proving the following theorem, which constitutes our analogue of Theorems~\ref{T:equicontinuity_H1} and \ref{T:equicontinuity_Hs} in the strongly confined case. We restrict attention to $1\leq s<2$, which includes the regularities $1\leq s\leq \frac32$ not covered by the existing literature.

\begin{theorem}[$\H^s$ bounds and equicontinuity for \eqref{KP-}]\label{T:confined} For each $1\leq s<2$, we have the following bound for $\H^\infty$ solutions to \eqref{KP-}:
\begin{equation}\label{Hs bdd confined}
\bigl\|(q(t),p(t))\bigr\|_{\H^s} \lesssim_s \bigl[1+2H_-(q_0,p_0)\bigr]\|(q_0,p_0)\|_{\H^s}
\end{equation}
for all times of existence.

Moreover, assume that $\Omega \subset \H^\infty$ is bounded and equicontinuous in $\H^s$ and
\begin{align}\label{bdd energy}
\sup_{(q,p)\in \Omega}\, H_-(q,p)<\infty. 
\end{align}
Then the set of orbits
$$
\Omega^*:= \bigl\{ (q(t),p(t))\,|\, (q,p)\text{ is an $\H^\infty$ solution to \eqref{KP-} with $(q(0), p(0))\in \Omega$}\bigr\}
$$
is also bounded and equicontinuous in $\H^s$.
\end{theorem}

\begin{proof}
In view of the constraint $\la q_0, Aq_0\ra<1$, the conservation of the strongly confined Hamiltonian \eqref{Hamiltonian} already implies a priori bounds in $\H^1$:
\begin{equation}\label{bdds confined}
\|p(t)\|^2\leq 2H_-(q_0,p_0) \qtq{and }\|A^{\frac12}q(t)\|^2 \leq \frac{2H_-(q_0,p_0)}{1+2H_-(q_0,p_0)} < 1
\end{equation}
for all times of existence. In particular, the speed function satisfies
\begin{align}\label{speed confined}
1\leq c_-(q(t)) \leq 1 + 2H_-(q_0,p_0).
\end{align}

The quantity $I_-(0;q,p)- I_-(z;q,p)$ is also coercive in the strongly confined case.  Indeed, using $\la q, \frac{A}{1+zA}q\ra \leq \la q,Aq\ra$ and the bounds \eqref{bdds confined} and \eqref{speed confined}, we find
\begin{align}\label{314}
&I_-(0;q,p)- I_-(z;q,p)\notag\\
&\quad= \langle p,\tfrac{zA}{1+zA}p\rangle+c_-(q)\langle q,\tfrac{zA^2}{1+zA}q\rangle -\langle q,\tfrac{A}{1+ zA}q\rangle\langle p,\tfrac{zA}{1+ zA}p\rangle+z\langle p,\tfrac{A}{1+ zA}q\rangle^2\notag\\
&\quad\geq \tfrac{1}{c_-(q)}\langle p,\tfrac{zA}{1+zA}p\rangle+c_-(q)\langle q,\tfrac{zA^2}{1+zA}q\rangle\notag\\
&\quad \geq \tfrac{1}{1 + 2H_-(q,p)}\Bigl\|\sqrt{\tfrac{zA}{1+zA}}\,p\Bigr\|^2+\Bigl\|\sqrt{\tfrac{zA}{1+zA}}A^{\frac12}q\Bigr\|^2.
\end{align}
On the other hand, an application of Cauchy--Schwarz,  \eqref{bdds confined}, and \eqref{speed confined} gives
\begin{align}\label{315}
I_-(0;q,p)- I_-(z;q,p)&\leq \Bigl[1+ \Bigl\|\sqrt{\tfrac{A}{1+zA}}q\Bigr\|^2 \Bigr]\Bigl\|\sqrt{\tfrac{zA}{1+zA}}\,p\Bigr\|^2 + c_-(q) \Bigl\|\sqrt{\tfrac{zA}{1+zA}}A^{\frac12}q\Bigr\|^2\notag\\
&\leq 2\Bigl\|\sqrt{\tfrac{zA}{1+zA}}\,p\Bigr\|^2 +  \bigl[1+ 2H_-(q,p)\bigr] \Bigl\|\sqrt{\tfrac{zA}{1+zA}}A^{\frac12}q\Bigr\|^2.
\end{align}

Setting $z=\kappa^{-2}$ in \eqref{314} and \eqref{315} and using the conservation of $I_-(0;q,p)- I_-(z;q,p)$ under the flow \eqref{KP-}, we deduce
\begin{align}\label{316}
\tfrac{1}{1 + 2H_-(q_0,p_0)}&\Bigl\|\sqrt{\tfrac{A}{\kappa^2+A}}\,p(t)\Bigr\|^2+\Bigl\|\sqrt{\tfrac{A}{\kappa^2+A}}A^{\frac12}q(t)\Bigr\|^2\notag\\
&\leq 2\Bigl\|\sqrt{\tfrac{A}{\kappa^2+A}}\,p_0\Bigr\|^2+\bigl[1+ 2H_-(q_0,p_0)\bigr]\Bigl\|\sqrt{\tfrac{A}{\kappa^2+A}}A^{\frac12}q_0\Bigr\|^2
\end{align}
for all times of existence.  Using
\begin{align}\label{318}
P_{\geq \kappa} \lesssim \tfrac{A}{\kappa^2 + A} \lesssim \tfrac1\kappa  + P_{\geq \sqrt \kappa}
\end{align}
in \eqref{316} yields $\H^1$ equicontinuity of orbits.

It remains to prove a priori bounds and equicontinuity of orbits for regularities $1<s<2$.  Integrating \eqref{316} against $\kappa^{2s-3}$ and using 
\begin{equation*}
\int_{k_0}^{\infty} \frac{\kappa^{2s-3}}{\kappa^2 + \lambda^2} \,d\kappa\sim_{s} (\kappa_0^2+\lambda^2)^{s-2},
\end{equation*}
we obtain
\begin{align}\label{317}
\tfrac{1}{1 + 2H_-(q_0,p_0)}&\bigl\|(\kappa_0^2+A)^{\frac{s-2}2}A^\frac12p(t)\bigr\|^2+\bigl\|(\kappa_0^2+A)^{\frac{s-2}2}A q(t)\bigr\|^2\notag\\
&\lesssim_s \bigl\|(\kappa_0^2+A)^{\frac{s-2}2}A^\frac12p_0\bigr\|^2+\bigl[1+ 2H_-(q_0,p_0)\bigr]\bigl\|(\kappa_0^2+A)^{\frac{s-2}2}Aq_0\bigr\|^2.
\end{align}
Taking $\kappa_0=0$ in \eqref{317} yields \eqref{Hs bdd confined}.

Boundedness of $\Omega^*$ in $\H^s$ follows from \eqref{Hs bdd confined}, \eqref{bdd energy}, and the boundedness of $\Omega$ in $\H^s$. To obtain equicontinuity of $\Omega^*$ in $\H^s$, we additionally use \eqref{318}, \eqref{317}, and the $\H^s$ boundedness and equicontinuity of $\Omega$. This completes the proof of Theorem~\ref{T:confined}.
\end{proof}

\subsection{An alternate approach at high regularity}\label{SS:4.2}

In the case $s\geq\frac32$, $\H^s$ control on solutions provides a bound on the time derivative of the speed function. As we will see in this subsection, this enables an alternate approach both to a priori bounds and to the propagation of equicontinuity, based on the Gronwall inequality.  One price to pay, however, is that it yields an exponential dependence on the time variable, unlike \eqref{Hs bdd confined} where the bound is uniform in time or \eqref{Hs bdd} where growth is at most linear.

While the approach is equally applicable to both the unconfined and strongly confined models, we will discuss only the strongly confined case; after all, Theorems~\ref{T:equicontinuity_H1} and \ref{T:equicontinuity_Hs} already cover the unconfined model for all $s\geq 1$.  

Consider for the moment the variant 
\begin{equation}\label{650}
\tfrac{d}{dt}q  = p, \qquad \tfrac{d}{dt}p = - \mkern2mu e^{2g(t)} A q 
\end{equation}
of our basic model, where we are imaging that the wave speed $e^{g(t)}$ has been prescribed in advance.  For any bounded measurable function $\phi:\R\to\R$, direct computation shows that 
\begin{align*}
\partial_t \Bigl\{ \bigl\| &\phi(A) p(t) \bigr\|^2 + e^{2g(t)} \bigl\| \phi(A) \sqrt{A}\, q(t) \bigr\|^2\Bigr\}
	\leq 2\mkern2mu|g'(t)| e^{2g(t)}\mkern2mu\bigl\| \phi(A) \sqrt{A}\, q(t) \bigr\|^2 .
\end{align*}
By the Gronwall inequality, this immediately yields
\begin{equation}\label{652}
\begin{aligned}
\sup_{|t|\leq T} \ &\Bigl\{ \bigl\| \phi(A) p(t) \bigr\|^2 + e^{2g(t)} \bigl\| \phi(A) \sqrt{A}\, q(t) \bigr\|^2\Bigr\} \\
	&\leq \exp\biggl\{ 2 \int_{-T}^T |g'(s)| \,ds \biggr\} \, \Bigl\{ \bigl\| \phi(A) p(0) \bigr\|^2 + e^{2g(0)} \bigl\| \phi(A) \sqrt{A}\, q(0) \bigr\|^2\Bigr\} .
\end{aligned}
\end{equation}

Needless to say, the utility of \eqref{652} is predicated on an effective bound for $g'(t)$.  In the case \eqref{KP-} under discussion, $ g'(t) = 2c_-(t) \langle p,A q \rangle$; consequently, by \eqref{speed confined},
\begin{equation}\label{653}
| g'(t) |  \leq [ 1 + 2H_{-}(p,q) ] \bigl\| \bigl(q(t),p(t)\bigr) \bigr\|_{\H^{3/2}}^2.
\end{equation}

\begin{corollary}\label{C:1.2}
For every $s\geq \frac32$, every $T>0$, and every set $\Omega$ of smooth initial data that is $\H^s$-bounded
and satisfies
\begin{align}\label{energy_-}
\sup_{(q,p) \in \Omega}\, H_-(q,p)<\infty,
\end{align}
the corresponding set of orbits under the \eqref{KP-} flow
\begin{equation}\label{C:1.2a}
 \Omega_T := \bigl\{ \bigl(q(t),p(t)\bigr) \big| \bigl(q(0),p(0)\bigr)\in\Omega \text{ and } |t|\leq T \bigr\}
\end{equation}
is $\H^s$-bounded.  If\/ $\Omega$ is also $\H^s$-equicontinuous, then so too is $\Omega_T$.
\end{corollary}

\begin{proof}
From Theorem~\ref{T:confined} we already have uniform $\H^{3/2}$ bounds on $\Omega_T$. Combining these with \eqref{speed confined}, \eqref{652}, and \eqref{653} we obtain
\begin{equation}\label{C:1.2b}
\begin{aligned}
\sup_{(q,p)\in\Omega_T} \ \bigl\| \bigl( \phi(A) q, \phi(A) & p \bigr)\bigr\|_{\H^1}^2 \leq C(\Omega,T) 
	\sup_{(q,p)\in\Omega} \ \bigl\| \bigl( \phi(A) q, \phi(A) p \bigr)\bigr\|_{\H^1}^2
\end{aligned}
\end{equation}
for any bounded measurable function $\phi$.  Importantly, the constant $C(\Omega,T)$ is independent of $\phi$.

As operators $\phi(A)$, we now consider the family
$$
\bigl(\tfrac{1 + A}{\kappa^2+\eta A}\bigr)^{s-1} \quad\text{indexed by $\eta>0$ and $\kappa\geq 1$.}
$$
To obtain $\H^s$ bounds, we choose $\kappa=1$ and send $\eta\to0$.  To propagate equicontinuity, we first choose $\kappa$ large depending on $\Omega$ and then send $\eta\to0$.
\end{proof}

\section{Global well-posedness in $\H^s$ for $s\geq \frac32$}\label{S:5}

This section is dedicated to the proof of the following
\begin{theorem}\label{T:GWP Hs}
For all $\frac 32\leq s\leq \infty$, both models \eqref{KP_pm} are globally well-posed in $\H^s$.
\end{theorem}

We will present the details in the unconfined case.  The method applies equally well in the strongly confined case; however, there are too many ultimately inconsequential changes to the details below to try to present both models in parallel.

Fix $s\geq \frac 32$ and initial data $(q_0,p_0)\in\H^s$. We will construct the solution to \eqref{KP+} with this initial data by proving convergence of a sequence of spectrally truncated approximations: For each $N\geq 1$, we define initial data
\begin{align*}
(q_{0,N}, p_{0,N}) = (P_{<N} q_0, P_{<N} p_0).
\end{align*}
By Lemma~\ref{L:GWP}, there exists a global $\H^\infty$ solution $(q_N, p_N)$ to \eqref{KP+} with this initial data. For any $T>0$,  Theorem~\ref{T:equicontinuity_Hs} yields
\begin{align}\label{221}
\sup_{N\geq 1}\,\sup_{t\in [-T,T]} \,\|(q_N(t),p_N(t))\|_{\H^s} \leq C(\|(q_0,p_0)\|_{\H^s})\la T\ra.
\end{align}

By definition,
\begin{align}\label{convg}
(q_{0,N}, p_{0,N}) \to (q_0, p_0) \qtq{in}  \H^s \qtq{as} N\to \infty.
\end{align}
Our first step is to show that $(q_N, p_N)$ forms a Cauchy sequence in a lower regularity space:

\begin{lemma}\label{L:F}
Fix $T>0$ and let $(q,p)$ and $(\tilde q, \tilde p)$ be two $C([-T,T]; \H^\infty)$ solutions to \eqref{KP+}.  Then the functional
\begin{align*}
F(q,p; \tilde q, \tilde p):=  \la p-\tilde p,  (1+A)^{-\frac12}(p-\tilde p)\ra+ \tfrac{c(q)^2+c(\tilde q)^2}2 \la q-\tilde q,  (1+A)^{-\frac12}A(q-\tilde q)\ra \end{align*}
satisfies
\begin{align}\label{F bdd}
\sup_{t\in[-T,T]}\, F(q(t),p(t); \tilde q(t), \tilde p(t))\lesssim F(q(0),p(0); \tilde q(0), \tilde p(0)),
\end{align}
where the implicit constant depends on $T$ and the $\H^{\frac32}$ norms of $(q(0),p(0))$ and $(\tilde q(0), \tilde p(0))$.
\end{lemma} 

\begin{proof}
A straightforward computation yields
\begin{align*}
\tfrac{d}{dt}F(q,p; \tilde q, \tilde p)
&= -\bigl[c^2(q)-c^2(\tilde q)\bigr] \langle p-\tilde p, (1+A)^{-\frac12}A(q+ \tilde q)\rangle \\
&\quad- 2\bigl[c(q)^3\la p, Aq\ra+ c(\tilde q)^3\la \tilde p, A\tilde q\ra\bigr] \langle q-\tilde q, (1+A)^{-\frac12}A(q-\tilde q)\ra\\
&= \bigl[c(q)+ c(\tilde q)\bigr] c(q) c(\tilde q) \la q-\tilde q, A(q+\tilde q)\ra  \langle p-\tilde p, (1+A)^{-\frac12}A(q+ \tilde q)\rangle\\
&\quad- 2\bigl[c(q)^3\la p, Aq\ra+c(\tilde q)^3\la \tilde p, A\tilde q\ra\bigr] \langle q-\tilde q, (1+A)^{-\frac12}A(q-\tilde q)\ra.
\end{align*}
Thus, by the Cauchy--Schwarz inequality we get 
\begin{align*}
\Bigl|\tfrac{d}{dt}F(q,p; \tilde q, \tilde p)\Bigr|
&\lesssim \|(1+A)^{\frac14}A^{\frac12}(q+\tilde q)\| \cdot \|(1+A)^{-\frac14}A(q+\tilde q)\| F(q,p; \tilde q, \tilde p)\\
&\quad +\bigl[\|(q,p)\|_{\H^{\frac32}}^2 +\|(\tilde q,\tilde p)\|_{\H^{\frac32}}^2\bigr] F(q,p; \tilde q, \tilde p)\\
&\lesssim \bigl[\|(q,p)\|_{\H^{\frac32}}^2 +\|(\tilde q,\tilde p)\|_{ \H^{\frac32}}^2\bigr] F(q,p; \tilde q, \tilde p).
\end{align*}
Now by Theorem~\ref{T:equicontinuity_Hs}, we have
\begin{align*}
\sup_{t\in [-T,T]} \,\|(q(t),p(t))\|_{\H^{\frac32}}& \leq C\bigl(\|(q_0,p_0)\|_{\H^{\frac32}}\bigr)\la T\ra \\
\sup_{t\in [-T,T]} \,\|(\tilde q(t),\tilde p(t))\|_{\H^{\frac32}}&\leq C(\|(\tilde q_0,\tilde p_0)\|_{\H^{\frac32}}\bigr)\la T\ra,
\end{align*}
and so claim \eqref{F bdd} follows from an application of the Gronwall inequality.
\end{proof}

Returning to the proof of Theorem~\ref{T:GWP Hs}, \eqref{221}, \eqref{convg}, and Lemma~\ref{L:F} imply
\begin{align}\label{401}
\lim_{N,M\to\infty}\ \sup_{t\in[-T,T]}\, F(q_N(t),p_N(t); q_M(t), p_M(t)) =0,
\end{align}
that is, $( (1+A)^{-\frac14}A^{\frac12}q_N, (1+A)^{-\frac14}p_N)$ is Cauchy in $C([-T,T]; X\times X)$.

Next, we employ Theorem~\ref{T:equicontinuity_Hs} to upgrade the convergence of $(q_N, p_N)$ to the space $C([-T,T]; \H^s)$. By the equicontinuity of orbits, we know that for each $\eps>0$ there exists $\kappa=\kappa(\eps,T)$  large enough so that
\begin{align}\label{402}
\sup_{N\geq 1} \, \sup_{t\in [-T,T]}\, \|(P_{\geq \kappa} q_N(t), P_{\geq \kappa} p_N(t))\|_{\H^s}<\eps.
\end{align}
On the other hand, we have the bounds
\begin{align*}
\|P_{<\kappa}A^{\frac12}(q_N -q_M)(t)\|_{H^{s-1}} &\lesssim  (1+\kappa)^{s-\frac12}\|(1+A)^{-\frac14}A^{\frac12}(q_N-q_M)(t)\|, \\
\|P_{<\kappa}(p_N -p_M)(t)\|_{H^{s-1}} &\lesssim  (1+\kappa)^{s-\frac12}\|(1+A)^{-\frac14}(p_N-p_M)(t)\|,
\end{align*}
and so \eqref{401} implies
$$
\sup_{t\in [-T,T]}\, \|P_{<\kappa}(q_N -q_M)(t), P_{<\kappa}(p_N -p_M)(t)\|_{\H^s}\to 0 \qtq{as} N, M\to \infty.
$$
Together with \eqref{402}, this shows that $(q_N, p_N)$ is Cauchy in $C([-T,T]; \H^s)$.

Let $(q,p)$ denote the limit of the sequence $(q_N, p_N)$.  It is easy to see that $(q,p)$ is an $\H^s$ solution to \eqref{KP+} on $[-T,T]$.  Also, in view of \eqref{221} it satisfies
\begin{align*}
\sup_{t\in [-T,T]} \,\|(q(t),p(t))\|_{\H^s} \leq C\bigl(\|(q_0,p_0)\|_{\H^s}\bigr)\la T\ra.
\end{align*}

The unconditional uniqueness of $C_t\H^s$ solutions follows from another application of Lemma~\ref{L:F}. Continuous dependence of the solution on the initial data in $\H^s$ also follows from Lemma~\ref{L:F} and Theorem~\ref{T:equicontinuity_Hs}, as in the argument above.  This completes the proof of well-posedness for all $\frac32\leq s< \infty$.  This then implies continuous dependence in $\H^\infty$ by virtue of the metric \eqref{Hinfty}. This completes the proof of Theorem~\ref{T:GWP Hs}.

\section{Existence of $\H^1$ solutions}\label{S:6}

In this section we prove the existence of solutions to \eqref{KP_pm} for initial data in the energy space $\H^1$. To avoid repetition, we once again present the details solely for the unconfined model \eqref{KP+}.

The first step is to show that the speeds along trajectories emanating from a bounded and equicontinuous subset of $\H^1$ form a precompact set of functions in the time variable. For our purposes, it suffices to consider smooth initial data, which leads to unique global solutions in view of Theorem~\ref{T:GWP Hs}.

\begin{lemma}\label{L:c equi}
Let $\Omega \subset \H^\infty$ be bounded and equicontinuous in $\H^1$. Then for each $T>0$, the set of functions
\begin{align*}
\mathcal C:=\bigl\{ c(q(t)):t\in [-T,T]\to \R \, |\, \text{$(q,p)$ solves \eqref{KP+} with $(q(0), p(0))\in \Omega$}\bigr\}
\end{align*}
is precompact in $C([-T, T])$.
\end{lemma}

\begin{proof}
The claim will follow from an application of the Arzela--Ascoli Theorem. Clearly, the set $\mathcal C$ is uniformly bounded by the constant $1$.  It thus remains to show that it is also equicontinuous.

By Theorem~\ref{T:equicontinuity_H1}, the set of orbits $\Omega_T$ emanating from the set $\Omega$ is bounded and equicontinuous in $\H^1$.  In particular, given $\eps>0$ there exists $\kappa=\kappa(\eps, \Omega)>0$ such that
\begin{align}\label{701}
\sup_{q\in \Omega_T} \la q, P_{\geq \kappa} A q\ra <\tfrac\eps4.
\end{align}

On the other hand, for an individual orbit we may compute
\begin{align}\label{702}
\tfrac{d}{dt} \la q, P_{< \kappa} A q\ra = 2\la p, P_{\geq \kappa} A q\ra\leq 2\kappa\|p\|\|A^{\frac12}q\|\leq C(\Omega)\la T\ra  \kappa,
\end{align}
where we used \eqref{H1 bdd} in the last step.

Applying \eqref{701} and \eqref{702}, for any solution $(q(t),p(t))$ to \eqref{KP+} with $(q(0),p(0))\in \Omega$ and any times $-T\leq t_1<t_2\leq T$, we may bound
\begin{align*}
\bigl| c(q(t_2)) -c(q(t_1))\bigr|
&\leq \bigl| \la q(t_2), A q(t_2) \ra - \la q(t_1), A q(t_1) \ra\bigr| \\
&\leq 2 \sup_{q\in \Omega_T} \la q, P_{\geq \kappa} A q\ra + \int_{t_1}^{t_2} \bigl|\tfrac{d}{dt} \la q(t), P_{< \kappa} A q(t)\ra\bigr|\, dt\\
&\leq \tfrac\eps2 + C(\Omega)\la T\ra \kappa |t_2-t_1|<\eps,
\end{align*}
provided $|t_2-t_1|$ is sufficiently small depending on $\Omega$ and $T$. This proves that the set $\mathcal C$ is equicontinuous.
\end{proof}

We are now ready to construct global solutions to \eqref{KP+} with initial data in the energy space.

\begin{theorem}\label{T:H1 exist}
Let $(q_0,p_0) \in \H^{1}$. For any $T>0$, there exists a solution $(q,p) \in C([-T,T];\H^1)$ to \eqref{KP+} with initial data $(q(0),p(0))= (q_0,p_0)$.
\end{theorem}

\begin{proof}
For each $N\geq 1$, let
\begin{align*}
(q_{0,N}, p_{0,N}) = (P_{<N} q_0, P_{<N} p_0)
\end{align*}
and let $(q_N, p_N)$ denote the global solution to \eqref{KP+} with data $(q_N(0), p_N(0))=(q_{0,N}, p_{0,N})$ guaranteed by Lemma~\ref{L:GWP}.

Invoking Lemma~\ref{L:c equi} and passing to a subsequence if necessary, we conclude that the functions $c_N=c(q_N):[-T,T]\to \R$ converge in $C([-T,T])$ to a continuous function $c:[-T,T]\to\R$.

The convergence of $(q_N,p_N)$ in $C([-T,T]; \H^1)$  is now a consequence of the following: 

\begin{lemma}\label{L:Cauchy}
Assume that the sequence of $\H^\infty$ initial data $\{(q_{0,n}, p_{0,n})\}_{n\geq 1}$ is Cauchy in $\H^1$ and that the corresponding sequence of speeds $c_n=c(q_n)$ is Cauchy in $C([-T,T])$ for some $T>0$.  Then the sequence of solutions $(q_n, p_n)$ to \eqref{KP+} with initial data $(q_n(0), p_n(0))= (q_{0,n}, p_{0,n})$ is Cauchy in $C([-T,T];\H^1)$.
\end{lemma}

\begin{proof}
By hypothesis, the set 
$$
 \Omega:=\bigl\{(q_{0,n}, p_{0,n}) \,|\, n\geq 1\bigr\} \quad \text{is bounded and equicontinuous in $\H^1$.}
$$
Consequently, by Theorem~\ref{T:equicontinuity_H1} the set of orbits
$$
\Omega_T:=\bigl\{ (q_n(t), p_n(t))\,|\, n\geq 1 \text{ and } t\in [-T,T]\bigr\}
$$
is bounded and equicontinuous in $\H^1$ with
\begin{align}\label{1131}
\sup_{(q,p)\in \Omega_T}\, \|(q,p)\|_{\H^1} \leq C(\Omega)\la T\ra.
\end{align}

As $\Omega_T$ is equicontinuous in $\H^1$, given $\eps>0$ there exists $\kappa=\kappa(\eps, \Omega)$ such that
\begin{align}\label{1132}
\sup_{(q,p)\in \Omega_T} \, \la p, P_{\geq \kappa} p\ra + \la q, P_{\geq \kappa} A q\ra <\eps.
\end{align}
For the low frequency component, we compute
\begin{align*}
\Bigl|\tfrac{d}{dt} \Bigl[&\|P_{<\kappa}(p_n-p_m)\|^2+ \|P_{<\kappa}A^{\frac12}(q_n-q_m)\|^2\Bigr]\Bigr|\\
&=2\Bigl|(1- c_n^2) \la p_n-p_m, P_{<\kappa} A (q_n-q_m)\ra + (c_m^2-c_n^2)\la p_n-p_m, P_{<\kappa}A q_m\ra\Bigr|\\
&\lesssim \kappa \|P_{<\kappa}(p_n-p_m)\| \|P_{<\kappa} A^{\frac12}(q_n-q_m)\| + |c_n-c_m| C(\Omega)\la T\ra \kappa, 
\end{align*}
where we used \eqref{1131} in the last step.  An application of the Gronwall inequality then yields
\begin{align*}
&\|P_{<\kappa}(p_n-p_m)(t))\|^2 + \|P_{<\kappa}A^{\frac12}(q_n-q_m)(t)\|^2\\
&\qquad\leq e^{C\kappa T}\Bigl[ \|P_{<\kappa}(p_{0,n}-p_{0,m})\|^2+ \|P_{<\kappa}A^{\frac12}(q_{0,n}-q_{0,m})\|^2\\
&\qquad\qquad\qquad  +C(\Omega)\la T\ra \kappa\int_0^t|c_n(\tau)-c_m(\tau)|\, d\tau \Bigr].
\end{align*}
This converges to zero as $n,m\to \infty$ uniformly for $t\in [-T,T]$; this is because $\{(q_{0,n}, p_{0,n})\}_{n\geq 1}$ is Cauchy in $\H^1$ and $c_n$ is Cauchy in $C([-T,T])$.  Combining this with \eqref{1132} to control the contribution of the high frequencies, we conclude that $(q_n, p_n)$ is Cauchy in $C([-T,T];\H^1)$.
\end{proof}

Returning to the proof of Theorem~\ref{T:H1 exist}, let $(q,p)\in C([-T,T]; \H^1)$ be the limit of $(q_N, p_N)$ guaranteed by  Lemma~\ref{L:Cauchy}. It is easy to see that this is a $C([-T,T]; \H^1)$ solution to \eqref{KP+} with data $(q(0),p(0))=(q_0,p_0)$.
\end{proof}

\section{Lessons from single-mode solutions}\label{S:7}

In this section, we restrict our attention to the unconfined case when $A$ is an unbounded operator with compact resolvent.  For example, $A$ could be the Dirichlet Laplacian on a bounded domain.  By examining solutions that live in a single eigenspace, we will be able to better understand the inherent limitations on a priori bounds for solutions to \eqref{KP+}.  In particular, we will show the essential role of both time-growth and equicontinuity of $\Omega$ in the a priori bounds \eqref{H1 bdd}.

Let $v$ be a normalized eigenvector of $\sqrt{A}$ with eigenvalue $\lambda>0$ and consider the solution to \eqref{KP+} with initial data
\begin{equation}\label{91}
q_0 =0 \qtq{and} p_0 = y_0 v
\end{equation}
where $y_0\in \R$. Evidently, such a solution can be written as
$$
q(t) = x(t) v, \quad p(t)=y(t) v,
$$
where $(x(t),y(t))$ is the solution to the system
\begin{equation}\label{92}
\dot x = y, \quad \dot y = - \tfrac{\lambda^2 x}{(1+\lambda^2 x^2)^2} \qtq{with initial data} x(0)=0, \quad y(0)=y_0.
\end{equation}

This system has Hamiltonian
$$
H = \tfrac12 y^2 + \tfrac12 - \tfrac12 \tfrac{1}{1+\lambda^2x^2}.
$$
As has long been known, such a planar system can be solved in terms of indefinite integrals and functional inversion.   There are two main cases: for $0<H < \frac12$ we have periodic orbits, while for $H>\frac12$ orbits escape ballistically to infinity. Both of these scenarios will be of interest to us.  The two exceptional situations are $H=0$, which corresponds to the zero solution, and $H=\frac12$, where $x(t)$ grows as $\sqrt t$.

We begin with the case of ballistic escape:

\begin{proposition}\label{P:93}
The solution to \eqref{92} with initial data $x_0=0$ and $y_0>1$ satisfies
\begin{align}\label{93}
t \smash[t]{\sqrt{y_0^2-1}}\leq x(t) \leq t y_0 \qtq{for all} t>0 .
\end{align}
\end{proposition}

\begin{proof}
With this initial data, $H=\tfrac12 y_0^2>\frac12$.  As this is conserved, we have
\begin{align}
\dot x^2  = \frac{y_0^2+(y_0^2-1)\lambda^2x^2}{1+\lambda^2x^2} 
\end{align}
for all time. Moreover, as $\dot x(0)=y_0>0$, by continuity we must have $\dot x(t)>0$ for all $t\in \R$.  Taking a square-root, we obtain a separable ODE and so deduce that 
\begin{align*}
t = \int_0^{x(t)} \sqrt{ \tfrac{1+\lambda^2u^2}{y_0^2+(y_0^2-1)\lambda^2u^2} }\;du   \qtq{for any}  t>0 .
\end{align*}
Noting that the integrand always lies in the interval $[y_0^{-1}, (y_0^2-1)^{-1/2}]$, we easily deduce \eqref{93}.
\end{proof}

Our first, rather trivial, application of Proposition~\ref{P:93} is to see that the time dependence in \eqref{H1 bdd} is sharp:

\begin{corollary}\label{C:93a}
If $\sqrt A$ has an eigenvalue $\lambda >0$, then there is an $\H^\infty$ solution $(q,p)$ to \eqref{KP+} that satisfies
\begin{align*}
\bigl\| \sqrt{A}\,q(t) \bigr\| \approx \la t\ra \qtq{uniformly for} t>0.
\end{align*}
\end{corollary}

\begin{proof}
Choose $y_0=\sqrt 2$ and let $(\tilde q,\tilde p)$ denote the solution with initial data \eqref{91}.  Now consider the solution $(q(t),p(t))=(\tilde q(t+1),\tilde p(t+1))$.  From Proposition~\ref{P:93}, we have
\begin{align*}
\bigl\| \sqrt A \, q(t) \bigr\|^2 &= \lambda^2 \tilde x(t+1)^2 \approx 1+t^2 \qtq{uniformly for} t>0. \qedhere
\end{align*}
\end{proof}

Our main reason for proving Proposition~\ref{P:93} is to demonstrate that bounds of the form \eqref{H1 bdd} can fail to hold if the assumption of equicontinuity on the initial data is relaxed to mere boundedness.

\begin{corollary}\label{C:93b}
If $A$ is unbounded with compact resolvent, then there is a sequence of $\H^\infty$ solutions $(q_n,p_n)$ to \eqref{KP+} and a sequence of times $t_n\to 0$ that satisfy 
\begin{align}\label{E:93b}
\bigl\| \bigl(q_n(0), p_n(0)\bigr) \bigr\|_{\H^1}^2 = 2 \qtq{and}
\lim_{n\to \infty}\, \bigl\| \bigl(q_n(t_n), p_n(t_n)\bigr) \bigr\|_{\H^1}^2= \infty .
\end{align}
\end{corollary}

\begin{proof}
Under our hypotheses, there is a sequence of eigenvalues $\lambda_n\to\infty$.  Let $v_n$ denote a corresponding sequence of normalized eigenvectors.  We choose $t_n=\lambda_n^{-1/2}$ and $(q_n,p_n)$ to be the solution with initial data $q_n(0)=0$ and $p_n(0)=\sqrt2\,v_n$. 

The first claim in \eqref{E:93b} is clear.  For the second, we employ \eqref{93}:
\begin{align}
\bigl\| \sqrt A \, q_n(t_n) \bigr\|^2 &= \lambda_n^2 x_n(t_n)^2 \geq  \lambda_n^2 t_n^2 = \lambda_n \to \infty \qtq{as} n\to \infty. \qedhere
\end{align}
\end{proof}

Let us now turn our attention to periodic solutions to \eqref{92}. Our reason for analyzing this case is to show that the data-to-solution map is not $\H^1$-uniformly continuous on bounded sets.

\begin{proposition}\label{P:94}
The solution to \eqref{92} with initial data $x_0=0$ and $0<y_0<1$ is periodic and the period $P(\lambda,y_0)$ satisfies
\begin{align}\label{94}
\tfrac{2\pi}{\lambda\sqrt{1-y_0^2}} \leq P(\lambda,y_0) &\leq \tfrac{2\pi}{\lambda(1-y_0^2)}
	\qtq{and} \tfrac{d}{dy_0} P(\lambda,y_0) \geq \tfrac{2\pi y_0}{\lambda(1-y_0^2)^{3/2}}.
\end{align}
\end{proposition}

\begin{proof}
Proceeding as in the proof of Proposition~\ref{P:93}, by conservation of energy, we find that $x(t)$ oscillates between $a$ and $-a$, where
\begin{align*}
a = \tfrac{y_0}{\lambda \sqrt{1-y_0^2}}.
\end{align*}
Moreover, the time to pass from one to the other is
\begin{align}\label{94a}
\tfrac12 P(\lambda,y_0) = \int_{-a}^a \sqrt{ \tfrac{1+\lambda^2u^2}{y_0^2-(1-y_0^2)\lambda^2u^2} }\;du
	= \tfrac{1}{\lambda\sqrt{1-y_0^2}} \int_{-1}^1 \sqrt{ \tfrac{1+\lambda^2a^2w^2}{1-w^2} }\;dw.
\end{align}
The second identity here corresponds to the change of variables $u=aw$.

While \eqref{94a} allows us to express the period in terms of the complete elliptic integral of the second kind, we only need elementary observations
about the function
\begin{align}\label{94b}
U:a\mapsto \int_{-1}^1 \sqrt{ \tfrac{1+\lambda^2a^2w^2}{1-w^2} }\;dw
\end{align}
in order to prove \eqref{94}.  Concretely, it is an increasing function of $a\in[0,\infty)$ and so also of $y_0\in[0,1)$; moreover,
\begin{align*}
\pi = \int_{-1}^1 \tfrac{dw}{\sqrt{1-w^2}}\leq U(a) &\leq \int_{-1}^1 \tfrac{\sqrt{1+a^2\lambda^2}\;dw}{\sqrt{1-w^2}} = \tfrac{\pi}{\sqrt{1-y_0^2}}. 
\end{align*}
The claims in \eqref{94} follow from these observations and \eqref{94a}.
\end{proof}

\begin{corollary}\label{C:95}
If $A$ is unbounded with compact resolvent, then there exist sequences of $\H^\infty$ solutions $(q_n,p_n)$ and $(\tilde q_n,\tilde  p_n)$ to \eqref{KP+}, as well as a sequence of times $t_n\to 0$ that satisfy
\begin{align}\label{95a}
\bigl\| \bigl(q_n(0), p_n(0)\bigr) \bigr\|_{\H^1} &= \tfrac12, \\
\label{95b}
\lim_{n\to \infty}\,\bigl\| \bigl(q_n(0)-\tilde q_n(0),\, p_n(0)-\tilde p_n(0)\bigr) \bigr\|_{\H^1}  &= 0, \\
\label{95c}
\lim_{n\to \infty}\,\bigl\| \bigl(q_n(t_n)-\tilde q_n(t_n),\, p_n(t_n)-\tilde p_n(t_n)\bigr) \bigr\|_{\H^1} &= 1.
\end{align}
\end{corollary}

\begin{proof}
Under our hypotheses, we may choose a sequence of eigenvalues $\lambda_n$ so that $\lambda_n/n\to\infty$.  Let $v_n$ denote a corresponding sequence of normalized eigenvectors.  We choose $(q_n,p_n)$ to be the solution with initial data $q_n(0)=0$ and $p_n(0)=\frac12 v_n$.  We then choose the sequence of (positive) times
\begin{equation*}
t_n := (n+\tfrac12) P(\lambda_n, \tfrac12) \leq (n+\tfrac12) \tfrac{8\pi}{3\lambda_n} \to 0 \quad\text{as $n\to\infty$.}
\end{equation*}
By choosing this multiple of the period, we find that $q_n(t_n)=0$ and $p_n(t_n)=-\frac12 v_n$.

Next we choose $\tilde y_n$ so that
$$
t_n = n P(\lambda_n, \tilde y_n) \qtq{or equivalently,}  P(\lambda_n, \tilde y_n) - P(\lambda_n, \tfrac12) = \tfrac{1}{2n} P(\lambda_n, \tfrac12).
$$
There is a unique such $\tilde y_n$ due to the fact that $y_0\mapsto P(\lambda_n, y_0)$ is an increasing, continuous, unbounded function.  Indeed, \eqref{94} guarantees that $\tilde y_n > \tfrac12$ and
\begin{equation*}
| \tilde y_n - \tfrac 12 | \leq  \tfrac{1}{2n} P(\lambda_n, \tfrac12) \tfrac{\lambda_n 3^{3/2}}{8\pi} \leq \tfrac{\sqrt{3}}{2n} .
\end{equation*} 

Now consider the sequence of solutions $(\tilde q_n,\tilde  p_n)$ with initial data $\tilde q_n(0)=0$ and $\tilde  p_n(0)=\tilde y_n v_n$.  Clearly,
\begin{equation*}
\bigl\| \bigl(q_n(0)-\tilde q_n(0),\, p_n(0)-\tilde p_n(0)\bigr) \bigr\|_{\H^1} = | \tilde y_n - \tfrac 12 | \to 0 \qtq{as} n\to \infty.
\end{equation*}

Recalling the relations between $t_n$ and the periods of our solutions, we find
\begin{equation*}
\bigl\| \bigl(q_n(t_n)-\tilde q_n(t_n),\, p_n(t_n)-\tilde p_n(t_n)\bigr) \bigr\|_{\H^1} = | \tilde y_n + \tfrac 12 | \to 1 \qtq{as} n\to \infty,
\end{equation*}
which completes the proof.
\end{proof}


\end{document}